\documentclass[11pt,letterpaper]{amsart}
\usepackage{amsmath,amssymb,amsthm,booktabs,longtable}
\usepackage{caption}
\usepackage[colorlinks=true,linkcolor=blue,citecolor=blue,urlcolor=blue]{hyperref}
\numberwithin{equation}{section}
\allowdisplaybreaks[2]
\newtheorem{theorem}{Theorem}[section]
\newtheorem{lemma}[theorem]{Lemma}
\newtheorem{proposition}[theorem]{Proposition}
\newtheorem{question}{Question}[section]
\newtheorem{problem}[question]{Problem}
\theoremstyle{definition}
\theoremstyle{remark}

\theoremstyle{plain}
\newcommand{\Cn}{\mathbb C^n}
\newcommand{\dd}{\,d}
\newcommand{\diag}{\operatorname{diag}}
\newcommand{\Hol}{\operatorname{Hol}}

\title[Zero-product problem for Toeplitz operators]{Zero-product problem for Toeplitz operators\\
on the Fock space}
\date{}
\author{Jie Qin}
\address{School of Mathematics and Statistics, Chongqing Technology and Business University, 400067, China}
\email{qinjie24520@163.com}
\subjclass[2020]{Primary 47B35, Secondary 30H20}
\keywords{Toeplitz operators, Zero-product problem, Fock space}
\begin{document}
\begin{abstract}
We answer Bauer and Le's question on zero products of Toeplitz
operators on the Fock space $F^2(\mathbb C^n)$[JFA, 261 (2011), 9, 2617--2640]. For $n\ge2$, we construct two bounded nonradial Schwartz
symbols on $\mathbb C^n$ whose Toeplitz operators are nonzero
and have zero product on $F^2(\mathbb C^n)$.  For $n=1$ and each $c\in(1/2,1)$, we
construct two smooth nonradial symbols of growth at most $Ce^{c|z|^2}$
for some constant $C>0$. Their extended Toeplitz operators in
$F^2(\mathbb C)$ are nonzero
and have zero product on all holomorphic polynomials. Moreover,  the second symbol is bounded when $c\geq3/4$. 
Our proofs use 
Gaussian kernel calculations, matrix identities, theta functions
and Fourier transform.
\end{abstract}

\thanks{The author was supported by the National Natural Science Foundation of China (12501154) and the Science and Technology Research Program of Chongqing Municipal Education Commission (KJQN202400806).}
\maketitle

\section{Introduction}

Let $n\ge1$ be an integer.
We write coordinate vectors as columns and use the superscripts
$T$ and $*$ for transpose and conjugate transpose, respectively. For $z=(z_1,\ldots,z_n)^T$ and
$w=(w_1,\ldots,w_n)^T$ in $\Cn$, set
$$
 z\cdot w=w^*z=\sum_{\ell=1}^n z_\ell\overline{w_\ell},
 \quad |z|=(z\cdot z)^{1/2}.
$$
Let $dV$ denote Lebesgue measure on $\Cn\simeq\mathbb R^{2n}$, and set
$$
 d\mu(z)=\pi^{-n}e^{-|z|^2}\dd V(z).
$$
The Fock space and its inner product are given by
$$
 F^2(\Cn)=L^2(\Cn,d\mu)\cap\Hol(\Cn),\quad
 \langle h,k\rangle=\int_{\Cn}h\overline{k}\dd\mu,
$$
where $\Hol(\Cn)$ denotes the space of entire functions. The reproducing kernels are $K_z(w)=e^{w\cdot z}$
and satisfy $\langle f,K_z\rangle=f(z)$ for every $f\in F^2(\Cn)$.

Let $P:L^2(\Cn,d\mu)\to F^2(\Cn)$ be the orthogonal projection. For a suitable measurable function $\varphi:\Cn\to\mathbb C$, the
Toeplitz operator $T_\varphi$ with the symbol $\varphi$ on $F^2(\Cn)$ is defined by $$T_\varphi:= P M_\varphi$$ where $M_\varphi$ denotes multiplication by $\varphi$. Its natural domain
consists of all functions $f\in F^2(\Cn)$ for which $\varphi f$
belongs to $L^2(\Cn,d\mu)$. For such $f$, the reproducing property gives
\begin{equation}\label{E4}
\begin{aligned}
(T_\varphi f)(z)
&=\langle P(\varphi f),K_z\rangle=\langle\varphi f,K_z\rangle\\
&=\int_{\Cn}\varphi(w)f(w)e^{z\cdot w}\,\dd\mu(w).
\end{aligned}
\end{equation}
If $\varphi$ is bounded, $T_\varphi$ is bounded on $F^2(\Cn)$.

The zero product problem asks whether $T_fT_g=0$ forces one of
the symbols to vanish almost everywhere. In their seminal work
\cite{BH}, Brown and Halmos
proved this implication for two essentially bounded symbols on
the Hardy space of the unit circle. The finite product version
can also be posed on other Hilbert spaces of holomorphic functions,
including the Bergman and Fock spaces.

\begin{problem}\label{prob:zero-product}
Let $N\ge2$ be an integer, and fix a Hardy, Bergman or Fock space.
Suppose $f_1,\ldots,f_N$ are  general  essentially bounded measurable
symbols whose Toeplitz operators satisfy
$$
 T_{f_1}\cdots T_{f_N}=0\quad\text{on this space}.
$$
Does it follow that one of these functions must be zero?
\end{problem}
The finite product problem
remained open for over thirty years on the Hardy space of the unit circle. Eventually,
Aleman and Vukoti\'c answered it affirmatively for every finite
$N$ in 2009 \cite{AV}.

On the Bergman space of the unit disk, the first results about the zero-product problem were obtained by Ahern and \v{C}u\v{c}kovi\'c \cite{AC2001}. They proved it for bounded harmonic
symbols and for bounded symbols with one radial
\cite{AC2004} on the Bergman space of the unit disk. Here a measurable symbol is radial if
it depends only on $|z|$ up to equality almost everywhere.
Le considered products of Toeplitz operators with bounded symbols,
all but possibly one radial. On the Bergman space of the unit disk
\cite{Le2010} and that of the unit ball in $\Cn$ with $n\ge2$
\cite{Le2009}, such a product has finite rank only if one symbol
vanishes almost everywhere. In particular, Le and Tikaradze treated bounded
pluriharmonic symbols $f,g$ on the unit ball and established
\cite{LeTikaradze2022}
$$
\operatorname{rank}(T_fT_g)<\infty
\quad\Longrightarrow\quad f=0\ \hbox{or}\ g=0.
$$
Their result gives the zero product implication for this class
of symbols on the unit ball. For two arbitrary bounded symbols,
the zero product problem remains open even on the unit disk
\cite{DasNarayanan2025}.

The Bergman space results above concern bounded symbols.
On the Fock space, allowing symbols with Gaussian growth
introduces an additional difficulty. The corresponding Toeplitz
operators need not be bounded, so the domains of their products
must also be specified \cite{BL}.
Bauer and Le studied the problem in the growth classes
\begin{equation}\label{E5}
 \mathcal D_c=\{u:\Cn\to\mathbb C\text{ measurable}:
               e^{-c|\cdot|^2}u\in L^\infty(\Cn)\},
 \quad c>0.
\end{equation}
For the unbounded symbols, we write $\widetilde T_u$ for the extension
in \cite[Definition 4.1]{BL}, which agrees with $T_u$ when $u$ is
bounded. Its definition in $F^2(\mathbb C)$ is recalled in
Section \ref{sec:moment-example}. Bauer and Le proved that
$\widetilde T_f\widetilde T_g p=0$ for every holomorphic polynomial
$p$ in $F^2(\Cn)$ implies $f=0$ or $g=0$ almost everywhere when
$f,g\in\mathcal D_c$ for some $c<1$ and one symbol is radial
\cite{BL}.

Two examples in the same paper show the role of the growth
condition and the number of factors. At $c=1$, Bauer and Le
constructed two nonzero unbounded radial symbols whose extended
Toeplitz operators have zero product \cite[Proposition 4.5]{BL}. They also constructed
three nonzero bounded radial symbols whose Toeplitz operators
have zero product \cite[Proposition 3.7]{BL}. The latter example
gives a negative answer to Problem \ref{prob:zero-product} for
three factors on the Fock space. Neither example addresses
the case of two bounded nonradial symbols.
For harmonic symbols in one complex variable whose entire parts
are bounded in modulus by $ce^{C|z|}$ for some constants $c,C>0$
independent of $z$, Bauer,
Choe and Koo proved that one symbol vanishes under the stronger
assumption that both ordered products vanish
\cite[Theorem 4.2]{Bauer2015}.

The following question, raised by Bauer and Le in
$\mathcal D_c$, includes the case of two bounded symbols without
a radiality assumption \cite[Question A]{BL}.

\begin{question}\label{q:BL}
Let $f,g\in\mathcal D_c$ for some  $0<c<1$ such that
$T_f T_g =0$. Is it true that $f=0$ or $g=0$ almost everywhere on $\Cn$?
\end{question}

In this paper, we give a negative answer to Question \ref{q:BL} for every $n\ge1$.
For $n\ge2$, we construct two nonradial Schwartz symbols. To specify this
class, we identify $z\in\Cn$ with
$$
 x=(\operatorname{Re}z_1,\operatorname{Im}z_1,\ldots,
    \operatorname{Re}z_n,\operatorname{Im}z_n)^T\in\mathbb R^{2n}.
$$
For $\alpha,\beta\in\mathbb N_0^{2n}$, where
$\mathbb N_0=\{0,1,2,\ldots\}$, write
$$
 x^\alpha=\prod_{\ell=1}^{2n}x_\ell^{\alpha_\ell},\quad
 \partial_x^\beta=
 \prod_{\ell=1}^{2n}\left(\frac{\partial}{\partial x_\ell}
                    \right)^{\beta_\ell}.
$$
The Schwartz class $\mathcal S(\mathbb C^n)$ consists of all
$u\in C^\infty(\mathbb C^n)$ such that
\begin{equation}\label{E6}
\sup_{x\in\mathbb R^{2n}}
|x^\alpha\partial_x^\beta u(x)|<\infty
\quad\text{for every }\alpha,\beta\in\mathbb N_0^{2n},
\end{equation}
where $\mathbb C^n$ is identified with $\mathbb R^{2n}$,
and smoothness and differentiation are understood in real coordinates.
A Schwartz function and all its real partial derivatives decay
faster than every negative power of $|x|$ as $|x|\to\infty$, see
\cite{Igusa2000}.
Our main result gives explicit symbols in this class.

\begin{theorem}\label{thm:counterexample}
Let $n\ge2$. For $z=(z_1,\ldots,z_n)^T\in\Cn$, set
\begin{equation}\label{eq:quadratic-forms}
 X=|z_1|^2-|z_2|^2,\quad
 Y=2\operatorname{Im}(\overline z_1z_2),\quad
 Z=2\operatorname{Re}(\overline z_1z_2).
\end{equation}
Define
\begin{equation}\label{eq:f}
 f(z)=2e^{-|z|^2}
 \left\{i e^{-2iY}\sin\bigl(2(X-Z)\bigr)
              -e^{2iY}\sin\bigl(2(X+Z)\bigr)\right\}
\end{equation}
and
\begin{equation} \label{eq:g}
 g(z)=2e^{-|z|^2}
 \left\{e^{-2iY}\cos\bigl(2(X+Z)\bigr)
              +i e^{2iY}\cos\bigl(2(X-Z)\bigr)\right\}.
\end{equation}
Then $f,g\in\mathcal S(\Cn)$ are nonradial and satisfy
$|f(z)|,|g(z)|\leq4e^{-|z|^2}$. The operators satisfy
$$
 T_fT_g=0,\quad T_f\ne0,\quad T_g\ne0
 \quad\hbox{on }F^2(\Cn).
$$
\end{theorem}

The symbols in Theorem \ref{thm:counterexample} belong to
$\mathcal D_0$ and have Gaussian decay. When $n\ge2$, the radiality
assumption in Bauer and Le's result cannot be replaced by smoothness
and rapid decay, even for bounded symbols.
Theorem \ref{thm:counterexample} gives a negative answer to
Problem \ref{prob:zero-product} already for $N=2$ on
$F^2(\Cn)$ when $n\ge2$. Both symbols are bounded, nonradial
and Schwartz. In the proof of  Theorem \ref{thm:counterexample}, we choose Gaussian
symbols whose Toeplitz operators send each kernel to a scalar
multiple of another kernel. Composition of these operators is
determined by ordered matrix products, so the construction reduces to
cancellation among terms with the same ordered matrix product.

The construction in Theorem \ref{thm:counterexample} uses the
two coordinates $z_1,z_2$ in \eqref{eq:quadratic-forms}, so it
requires $n\ge2$. On the complex line $z=(z_1,0,\ldots,0)^T$,
we have $X=|z_1|^2$ and $Y=Z=0$, and the restrictions of both
symbols depend only on $|z_1|$. To obtain nonradial symbols
when $n=1$, we use a different construction.

\begin{theorem}\label{thm:moment-example}
Let $n=1$ and $c\in(1/2,1)$. There are nonzero smooth nonradial
functions $f,g\in\mathcal D_c$ whose extended Toeplitz operators
in $F^2(\mathbb C)$ are nonzero. Their domains contain the space
$\mathcal P=\mathbb C[z]$ of all holomorphic polynomials on
$\mathbb C$. Moreover,
\begin{equation}\label{eq:moment-product}
 \mathcal P\subset\mathcal D(\widetilde T_f\widetilde T_g),
 \quad
 \widetilde T_f\widetilde T_g p=0
 \quad\hbox{in }F^2(\mathbb C)\quad(p\in\mathcal P).
\end{equation}
When $c\ge3/4$, the symbol $g$ can be chosen bounded.
\end{theorem}

Theorem \ref{thm:moment-example} uses the formulation on
holomorphic polynomials preceding Question A in
\cite[Section 6]{BL}. In Section \ref{sec:moment-example},
we use a Jacobi theta function to construct an entire function
whose Gaussian weighted modulus is periodic. A smooth cutoff in the Fourier variables then yields
integral identities that make the operator product vanish.

In this paper, we consider the zero-product problem for Toeplitz operators on the Fock space. The Theorems \ref{thm:counterexample} and
\ref{thm:moment-example} resolve Bauer and Le's question
in every complex dimension.
For $n\ge2$, they also give a negative answer to
Problem \ref{prob:zero-product} on $F^2(\Cn)$ for every $N\ge2$.
The case $N=2$ is given by Theorem \ref{thm:counterexample},
and further factors may have the constant symbol $1$.
For $n=1$, Theorem \ref{thm:moment-example} gives examples
with $g$ bounded when $c\ge3/4$, but does not give two bounded
symbols. Whether two nonzero bounded symbols can have a zero
Toeplitz product on $F^2(\mathbb C)$ remains open.

For the matrix calculations below, we write $I_d$ for the
$d\times d$ identity matrix and $\diag$ for the diagonal or
block diagonal matrix with the indicated entries or blocks.
For $u,v\in\mathbb C^d$, we use the same inner product convention
$u\cdot v=v^*u$.

\section{Gaussian symbols}

We use the following special case of the complex Gaussian
integral formula \cite[Proposition III.2]{JohnParthasarathy2021}.
For $\rho,\tau,\lambda\in\mathbb C$ with
$\operatorname{Re}\lambda>0$,
\begin{equation}\label{eq:scalar}
  \pi^{-1}\int_{\mathbb C}
  e^{-\lambda|w|^2+\rho w+\tau\overline w}\dd A(w)
  =\lambda^{-1}e^{\rho\tau/\lambda},
\end{equation}
where $dA$ denotes area measure on $\mathbb C$.
After a unitary change of variables, \eqref{eq:scalar} gives
the action of the Toeplitz operators with the following Gaussian
symbols on reproducing kernels. In the block notation below,
when $n=2$, vectors in $\mathbb C^{n-2}$ have no coordinates and
all blocks of size $n-2$ are omitted. The squared norms and inner products of these vectors are zero,
and determinants of empty blocks equal $1$.

\begin{lemma}\label{lem:gaussian}
Let $n\ge2$, and let $U$ be a $2\times2$ unitary matrix satisfying
$U+U^*=I_2$ and $\det U=1$. Write
$$
 z=\begin{pmatrix}\xi\\z'\end{pmatrix},\quad
 \xi=(z_1,z_2)^T\in\mathbb C^2,\quad
 z'=(z_3,\ldots,z_n)^T\in\mathbb C^{n-2}.
$$
Set
\begin{equation}\label{eq:gaussian-symbol}
\sigma_U(z)=\exp\bigl(((I_2-4U^*)\xi)\cdot\xi-|z'|^2\bigr)
\end{equation}
and
\begin{equation}\label{eq:D-definition}
  D_U=\diag(U/4,I_{n-2}/2),\quad \kappa_n=2^{-n-2}.
\end{equation}
Then $|\sigma_U(z)|=e^{-|z|^2}$ and, for every $a,z\in\Cn$,
\begin{equation}\label{eq:kernel-gaussian}
 (T_{\sigma_U}K_a)(z)
 =\kappa_n e^{(D_Uz)\cdot a}
 =\kappa_n K_{D_U^*a}(z).
\end{equation}
\end{lemma}

\begin{proof}
To compute the modulus of $\sigma_U$, set $A=I_2-4U^*$.
Using $A^*=I_2-4U$, $\overline{\xi^*A\xi}=\xi^*A^*\xi$ and
$U+U^*=I_2$, we obtain
$$
\begin{aligned}
 \operatorname{Re}\bigl((A\xi)\cdot\xi\bigr)
 &=\frac12\bigl(\xi^*A\xi+\xi^*A^*\xi\bigr)\\
 &=\xi^*\bigl(I_2-2(U+U^*)\bigr)\xi\\
 &=-\xi^*\xi=-|\xi|^2.
\end{aligned}
$$
Substituting this real part into \eqref{eq:gaussian-symbol} gives
$$
\begin{aligned}
 |\sigma_U(z)|
 &=\exp\bigl(\operatorname{Re}((A\xi)\cdot\xi)-|z'|^2\bigr)\\
 &=\exp\bigl(-|\xi|^2-|z'|^2\bigr)=e^{-|z|^2}.
\end{aligned}
$$
This bound also shows that $T_{\sigma_U}$ is bounded on
$F^2(\Cn)$, so it acts on every kernel $K_a$.

To compute this action, we combine the symbol with the Gaussian
weight by setting $M_U=\diag(4U^*,2I_{n-2})$.
The assumptions on $U$, together with \eqref{eq:D-definition}, give
\begin{equation}\label{eq:matrix-data}
 M_U^{-1}=D_U,\quad
 \det M_U=4^2\det(U^*)\,2^{n-2}=2^{n+2},\quad
 \frac{M_U+M_U^*}{2}=2I_n.
\end{equation}
Here the inverse uses $U^{-1}=U^*$, and the determinant uses
$\det(U^*)=\overline{\det U}=1$. To express the weighted symbol
in terms of $M_U$, write
$$
 w=\begin{pmatrix}\eta\\w'\end{pmatrix},\quad
 \eta=(w_1,w_2)^T\in\mathbb C^2,\quad
 w'=(w_3,\ldots,w_n)^T\in\mathbb C^{n-2}.
$$
With $|w|^2=|\eta|^2+|w'|^2$, the definition
\eqref{eq:gaussian-symbol} yields
$$
\begin{aligned}
 \sigma_U(w)e^{-|w|^2}
 &=\exp\bigl(((I_2-4U^*)\eta)\cdot\eta
             -|w'|^2-|\eta|^2-|w'|^2\bigr)\\
 &=\exp\bigl(|\eta|^2-4(U^*\eta)\cdot\eta
             -|\eta|^2-2|w'|^2\bigr)\\
 &=\exp\bigl(-4(U^*\eta)\cdot\eta-2|w'|^2\bigr).
\end{aligned}
$$
The quadratic form in the last exponent is determined by $M_U$,
because its block form gives
$$
 (M_Uw)\cdot w
 =4(U^*\eta)\cdot\eta+2|w'|^2.
$$
Combining these expressions for the weighted symbol and applying
\eqref{E4} to $K_a$, we obtain
\begin{align}\label{eq:gaussian-integral}
 (T_{\sigma_U}K_a)(z)
 &=\pi^{-n}\int_{\Cn}
 \sigma_U(w)e^{w\cdot a+z\cdot w-|w|^2}\dd V(w)\nonumber\\
 &=\pi^{-n}\int_{\Cn}
 e^{-(M_Uw)\cdot w+w\cdot a+z\cdot w}\dd V(w).
\end{align}

To evaluate \eqref{eq:gaussian-integral} by separating the
coordinates, we first verify absolute convergence. Fix
$a,z\in\Cn$ and put $L=|a|+|z|$. The Hermitian part in
\eqref{eq:matrix-data} controls the quadratic term through
$$
 \operatorname{Re}\bigl((M_Uw)\cdot w\bigr)
 =w^*\frac{M_U+M_U^*}{2}w
 =2w^*w=2|w|^2.
$$
The remaining linear terms satisfy
$\operatorname{Re}(w\cdot a+z\cdot w)\le L|w|$.
To absorb this bound into the Gaussian decay, put $r=|w|$ and
complete the square to obtain
$$
 -2r^2+Lr
 =-r^2-\left(r-\frac L2\right)^2+\frac{L^2}{4}
 \leq-r^2+\frac{L^2}{4}.
$$
Combining the quadratic and linear bounds gives
\begin{align}\label{E17}
 \left|e^{-(M_Uw)\cdot w+w\cdot a+z\cdot w}\right|
 &=e^{-2|w|^2+\operatorname{Re}(w\cdot a+z\cdot w)}\nonumber\\
 &\leq e^{-2|w|^2+L|w|}
 \leq e^{L^2/4}e^{-|w|^2}.
\end{align}
The majorant in \eqref{E17} has integral $\pi^ne^{L^2/4}$
over $\Cn$, which proves absolute convergence of
\eqref{eq:gaussian-integral} for every fixed $a,z$.

We next diagonalize $M_U$ by a unitary change of coordinates.
The identities $U^*U=UU^*=I_2$ show that $M_U$ is normal, since
$$
 M_UM_U^*
 =\diag(16U^*U,4I_{n-2})
 =\diag(16I_2,4I_{n-2})
 =M_U^*M_U.
$$
The spectral theorem therefore gives a unitary matrix $Q$ such that
$$
 Q^*M_UQ=\Lambda,\quad
 \Lambda=\diag(\lambda_1,\ldots,\lambda_n).
$$
To check the hypothesis of \eqref{eq:scalar} for the diagonal
entries, we take Hermitian parts and use \eqref{eq:matrix-data} to obtain
$$
 \diag(\operatorname{Re}\lambda_1,\ldots,
       \operatorname{Re}\lambda_n)
 =\frac{\Lambda+\Lambda^*}{2}
 =Q^*\frac{M_U+M_U^*}{2}Q
 =2I_n.
$$
Then $\operatorname{Re}\lambda_\ell=2>0$ for every $\ell$, as
required in \eqref{eq:scalar}. In particular, each $\lambda_\ell$
is nonzero. We use the corresponding coordinates
$v=Q^*w$, $b=Q^*a$ and $y=Q^*z$ in $\Cn$, so that
$w=Qv$, $a=Qb$ and $z=Qy$. For $1\leq\ell\leq n$,
the entries $v_\ell$, $b_\ell$ and $y_\ell$ are the $\ell$th
coordinates of the column vectors $v$, $b$ and $y$.
The unitary change of coordinates preserves Lebesgue measure
and the Hermitian inner product, so we obtain 
$$
 (M_Uw)\cdot w=\sum_{\ell=1}^n\lambda_\ell|v_\ell|^2,
 \quad
 w\cdot a=\sum_{\ell=1}^n v_\ell\overline{b_\ell},
 \quad
 z\cdot w=\sum_{\ell=1}^n y_\ell\overline{v_\ell}.
$$
With these expressions, Fubini's theorem and the integrable
majorant \eqref{E17} factor \eqref{eq:gaussian-integral} into
one integral for each coordinate. Applying \eqref{eq:scalar}
with $\lambda=\lambda_\ell$, $\rho=\overline{b_\ell}$ and
$\tau=y_\ell$ to each factor gives
\begin{align*}
 (T_{\sigma_U}K_a)(z)
 &=\prod_{\ell=1}^n\left\{
   \pi^{-1}\int_{\mathbb C}
   e^{-\lambda_\ell|v_\ell|^2+v_\ell\overline{b_\ell}
                         +y_\ell\overline{v_\ell}}\dd A(v_\ell)
                          \right\}\\
 &=\left(\prod_{\ell=1}^n\lambda_\ell^{-1}\right)
      \exp\left(\sum_{\ell=1}^n
                      \frac{y_\ell\overline{b_\ell}}{\lambda_\ell}
          \right).
\end{align*}

To return to the original coordinates, note that
$\prod_{\ell=1}^n\lambda_\ell=\det M_U$ and
$M_U^{-1}=Q\Lambda^{-1}Q^*$. The latter identity rewrites the exponent as
\begin{align*}
 \sum_{\ell=1}^n\lambda_\ell^{-1}y_\ell\overline{b_\ell}
 &=b^*\Lambda^{-1}y\\
 &=(Q^*a)^*\Lambda^{-1}(Q^*z)\\
 &=a^*Q\Lambda^{-1}Q^*z
 =a^*M_U^{-1}z=(M_U^{-1}z)\cdot a.
\end{align*}
Substituting this exponent and using \eqref{eq:matrix-data}
and \eqref{eq:D-definition}, we obtain
$$
 (T_{\sigma_U}K_a)(z)
 =\frac1{\det M_U}\exp\bigl((M_U^{-1}z)\cdot a\bigr)
 =\kappa_n e^{(D_Uz)\cdot a}.
$$
Finally, $(D_Uz)\cdot a=z\cdot(D_U^*a)$, so the last expression
is also $\kappa_n K_{D_U^*a}(z)$, as asserted in
\eqref{eq:kernel-gaussian}.
\end{proof}

\section{Proof of Theorem \ref{thm:counterexample}}

Throughout this section, let $n\ge2$. We first construct the matrices required by Lemma \ref{lem:gaussian}
and identify the resulting Gaussian sums with the symbols in the theorem.
We then verify the matrix cancellations and use them to compute the
Toeplitz product.

Set
\begin{equation}\label{eq:E-definition}
 E_1=\begin{pmatrix}i&0\\0&-i\end{pmatrix},\quad
 E_2=\begin{pmatrix}0&1\\-1&0\end{pmatrix},\quad
 E_3=\begin{pmatrix}0&i\\i&0\end{pmatrix}.
\end{equation}
These matrices satisfy $E_j^*=-E_j$, and direct multiplication gives
\begin{equation}\label{eq:E-table}
 \begin{array}{c|rrr}
       &E_1&E_2&E_3\\\hline
 E_1&-I_2&E_3&-E_2\\
 E_2&-E_3&-I_2&E_1\\
 E_3&E_2&-E_1&-I_2
 \end{array}
\end{equation}
The entry in row $E_i$ and column $E_j$ is the product $E_iE_j$.
For example, the table gives $E_1E_2=E_3$ and $E_2E_1=-E_3$.
Its diagonal entries give $E_j^2=-I_2$, and the entries outside
the diagonal give $E_jE_k+E_kE_j=0$ for $j\ne k$.
These identities determine the square of the following matrix.
For a choice of signs $\varepsilon=(\varepsilon_1,\varepsilon_2,
\varepsilon_3)\in\{-1,1\}^3$, set
\begin{equation}\label{eq:W-definition}
 S_\varepsilon=\varepsilon_1E_1+\varepsilon_2E_2
                         +\varepsilon_3E_3
 \quad\text{and}\quad
 W(\varepsilon_1,\varepsilon_2,\varepsilon_3)
       =\tfrac12(I_2+S_\varepsilon).
\end{equation}
With $W=W(\varepsilon_1,\varepsilon_2,\varepsilon_3)$,
\eqref{eq:E-definition} and \eqref{eq:W-definition} give
\begin{equation}\label{eq:W-entries}
 W=\frac12\begin{pmatrix}
       1+i\varepsilon_1&\varepsilon_2+i\varepsilon_3\\
       -\varepsilon_2+i\varepsilon_3&1-i\varepsilon_1
     \end{pmatrix}.
\end{equation}
To use these matrices in the construction, we verify the
hypotheses of Lemma \ref{lem:gaussian} and compute the corresponding
Gaussian symbols.

\begin{lemma}\label{lem:matrix-family}
For every $\varepsilon\in\{-1,1\}^3$, the matrix
$W=W(\varepsilon_1,\varepsilon_2,\varepsilon_3)$ is unitary and satisfies
$W+W^*=I_2$ and $\det W=1$. With $X,Y,Z$ as in
\eqref{eq:quadratic-forms}, its Gaussian symbol satisfies
\begin{equation}\label{eq:phase}
 \sigma_{W(\varepsilon_1,\varepsilon_2,\varepsilon_3)}(z)
 =e^{-|z|^2}e^{2i(\varepsilon_1X+\varepsilon_2Y+
                                      \varepsilon_3Z)}.
\end{equation}
\end{lemma}

\begin{proof}
From \eqref{eq:E-table}, one can see that
$$
 S_\varepsilon^*=-S_\varepsilon,\quad
 S_\varepsilon^2
 =\sum_{j=1}^3\varepsilon_j^2E_j^2
  +\sum_{j<k}\varepsilon_j\varepsilon_k(E_jE_k+E_kE_j)
 =-3I_2.
$$
Combining these identities with \eqref{eq:W-definition}, we obtain
$$
 WW^*=W^*W=\tfrac14(I_2-S_\varepsilon^2)=I_2,
 \quad W+W^*=I_2,\quad \det W=1.
$$
These identities verify that $W$ satisfies the hypotheses of
Lemma \ref{lem:gaussian} for every choice of signs.

For the column blocks $\xi$ and $z'$ of $z$ from
Lemma \ref{lem:gaussian}, equations
 \eqref{eq:quadratic-forms} and \eqref{eq:E-definition} give
\begin{equation*}
 (E_1\xi)\cdot\xi=i(|z_1|^2-|z_2|^2)=iX,\quad
 (E_2\xi)\cdot\xi=\overline z_1z_2-\overline z_2z_1=iY,
\end{equation*}
and 
\begin{align*}
 (E_3\xi)\cdot\xi&=i(\overline z_1z_2+\overline z_2z_1)=iZ.
\end{align*}
Since $I_2-4W^*=-I_2+2S_\varepsilon$, substitution of these
three identities into \eqref{eq:gaussian-symbol} proves \eqref{eq:phase}.
\end{proof}

We now choose eight matrices from this family to represent
$f$ and $g$ as sums of Gaussian symbols. Each row of the following
array defines two matrices $U_j,V_j$ and their respective scalar
coefficients $c_j,d_j$, where $j\in\{1,2,3,4\}$. The three
arguments of $W$ give the signs multiplying $E_1,E_2,E_3$,
in that order.

\begin{equation}\label{eq:UV}
 \begin{array}{c|cc|cc}
 j&U_j&V_j&c_j&d_j\\\hline
 1&W(-1,-1,-1)&W(1,-1,-1)&1&1\\
 2&W(1,1,-1)&W(-1,1,-1)&i&-i\\
 3&W(1,-1,1)&W(-1,-1,1)&1&-1\\
 4&W(-1,1,1)&W(1,1,1)&i&i
 \end{array}
\end{equation}
For example, the second row gives
\begin{equation*}
 U_2=W(1,1,-1)=\tfrac12(I_2+E_1+E_2-E_3),\quad c_2=i,
\end{equation*}
and
\begin{align*}
 V_2&=W(-1,1,-1)=\tfrac12(I_2-E_1+E_2-E_3),\quad d_2=-i.
\end{align*}
The coefficients $c_j$ and $d_j$ occur in the representations
$$
g=\sum_{j=1}^{4}c_j\sigma_{U_j},
\quad
f=\sum_{j=1}^{4}d_j\sigma_{V_j}.
$$
So, $c_j$ multiplies $\sigma_{U_j}$ in the expression for $g$,
and $d_j$ multiplies $\sigma_{V_j}$ in the expression for $f$.
For example, the second row gives $c_2=i$ and $d_2=-i$,
then the corresponding terms are $i\sigma_{U_2}$ in $g$
and $-i\sigma_{V_2}$ in $f$. The next lemma verifies these
representations and establishes the regularity, bounds and
nonradiality required in Theorem \ref{thm:counterexample}.

\begin{lemma}\label{lem:symbol-properties}
The functions $f,g$ in \eqref{eq:f} and \eqref{eq:g} have the representations
\begin{equation}\label{eq:symbol-sums}
 f=\sum_{k=1}^4d_k\sigma_{V_k},\quad
 g=\sum_{j=1}^4c_j\sigma_{U_j}.
\end{equation}
Then $f,g$ belong to $\mathcal S(\Cn)$, are nonradial even up to equality
almost everywhere, and satisfy
$$
 |f(z)|\leq4e^{-|z|^2},\quad
 |g(z)|\leq4e^{-|z|^2}.
$$
\end{lemma}

\begin{proof}
By \eqref{eq:phase} and \eqref{eq:UV},
\begin{align*}
 &\qquad e^{|z|^2}\sum_{k=1}^4d_k\sigma_{V_k}(z)\\
 &=e^{2i(X-Y-Z)}-i e^{2i(-X+Y-Z)}-e^{2i(-X-Y+Z)}+i e^{2i(X+Y+Z)}\\
 &=e^{-2iY}\bigg\{e^{2i(X-Z)}-e^{-2i(X-Z)}\bigg\}+i e^{2iY}\bigg\{e^{2i(X+Z)}-e^{-2i(X+Z)}\bigg\}\\
 &=2i e^{-2iY}\sin\bigl(2(X-Z)\bigr)
       -2e^{2iY}\sin\bigl(2(X+Z)\bigr).
\end{align*}
For the other sum, \eqref{eq:UV} and \eqref{eq:phase} give
$$
\begin{aligned}
 &\qquad e^{|z|^2}\sum_{j=1}^4c_j\sigma_{U_j}(z)\\
 &=e^{2i(-X-Y-Z)}+i e^{2i(X+Y-Z)}+e^{2i(X-Y+Z)}+i e^{2i(-X+Y+Z)}\\
 &=e^{-2iY}\bigg\{e^{-2i(X+Z)}+e^{2i(X+Z)}\bigg\}+i e^{2iY}\bigg\{e^{2i(X-Z)}+e^{-2i(X-Z)}\bigg\}\\
 &=2e^{-2iY}\cos\bigl(2(X+Z)\bigr)
       +2i e^{2iY}\cos\bigl(2(X-Z)\bigr).
\end{aligned}
$$
Comparing these two expansions with \eqref{eq:f} and \eqref{eq:g},
we obtain \eqref{eq:symbol-sums}.
All coefficients in \eqref{eq:UV} have modulus one, so
\eqref{eq:symbol-sums} and Lemma \ref{lem:gaussian} give the stated estimates.
In particular, $T_f$ and $T_g$ are bounded on $F^2(\Cn)$.

By \eqref{eq:phase}, each Gaussian symbol in
\eqref{eq:symbol-sums} has the form $e^{\Phi(x)}$ in real
coordinates, where
$$
 \Phi(x)=-|x|^2
 +2i\bigl(\varepsilon_1X(x)+\varepsilon_2Y(x)+\varepsilon_3Z(x)\bigr)
$$
for the corresponding signs $\varepsilon\in\{-1,1\}^3$.
Here $X(x),Y(x),Z(x)$ are the quadratic forms in
\eqref{eq:quadratic-forms} written in real coordinates.
Since these forms are real valued, $\Phi$ is a quadratic
polynomial with complex coefficients and
$\operatorname{Re}\Phi(x)=-|x|^2$. Repeated differentiation gives
$$
 \partial_x^\beta e^{\Phi(x)}=R_\beta(x)e^{\Phi(x)}
$$
for a polynomial $R_\beta$, since $R_0=1$ and each differentiation
in $x_\ell$ replaces the polynomial $R$ by
$\partial_{x_\ell}R+R\partial_{x_\ell}\Phi$, which is again a
polynomial. For each pair of multiindices
$\alpha,\beta\in\mathbb N_0^{2n}$, there exist a constant $C>0$
and an integer $N\geq0$, depending on $n,\alpha,\beta$ and $\Phi$
but not on $x$, such that
$$
 |x^\alpha\partial_x^\beta e^{\Phi(x)}|
 =|x^\alpha R_\beta(x)|e^{-|x|^2}
 \leq C(1+|x|)^N e^{-|x|^2}.
$$
The right side is bounded on $\mathbb R^{2n}$, so the criterion
\eqref{E6} shows that each term in \eqref{eq:symbol-sums} is
Schwartz, as are $f$ and $g$.

To prove nonradiality, set
$$
 r=\sqrt{\pi/8},\quad p=(r,0,\ldots,0)^T,\quad
 q=(r/\sqrt2,ir/\sqrt2,0,\ldots,0)^T.
$$
Then $|p|^2=|q|^2=\pi/8<1$, and
$$
 (X(p),Y(p),Z(p))=(\pi/8,0,0),\quad
 (X(q),Y(q),Z(q))=(0,\pi/8,0).
$$
Substitution in \eqref{eq:f} and \eqref{eq:g} gives
$$
 |f(p)|=|g(p)|=2e^{-\pi/8}>0,\quad f(q)=g(q)=0.
$$
If $f$ were radial almost everywhere, continuity would imply
that $f(z)$ depends only on $|z|$. Hence $f(p)=f(q)$, since
$|p|=|q|$, contradicting the values above. The same argument
applies to $g$, so neither symbol is radial, even up to equality
almost everywhere, which completes the proof.
\end{proof}

We next compute the ordered products of the matrices in \eqref{eq:UV}.
For $s=(s_1,s_2,s_3),t=(t_1,t_2,t_3)\in\{-1,1\}^3$,
\eqref{eq:W-definition} and \eqref{eq:E-table} give
\begin{align}
 4W(s_1,s_2,s_3)W(t_1,t_2,t_3)
 &=(1-s_1t_1-s_2t_2-s_3t_3)I_2\nonumber\\
 &\quad+(s_1+t_1+s_2t_3-s_3t_2)E_1\nonumber\\
 &\quad+(s_2+t_2+s_3t_1-s_1t_3)E_2\nonumber\\
 &\quad+(s_3+t_3+s_1t_2-s_2t_1)E_3.
 \label{eq:W-product}
\end{align}
Write $A_0,A_1,A_2,A_3$ for the respective coefficients of
$I_2,E_1,E_2,E_3$ in \eqref{eq:W-product}. For example, $U_1V_1$ corresponds to
$s=(-1,-1,-1)$ and $t=(1,-1,-1)$, so we have 
\begin{align*}
 A_0&=1-(-1)(1)-(-1)(-1)-(-1)(-1)=0.\end{align*}
 The same way gives  $A_1=0, A_2=-4,$ and $A_3=0$.
Then $4U_1V_1=-4E_2$. 

To collect terms with the same matrix product, set
\begin{equation}\label{E7}
 \mathcal H=\{I_2,E_1,E_2,E_3,-E_1,-E_2,-E_3\}.
\end{equation}
For each $H\in\mathcal H$, define
\begin{equation}\label{eq:cancellation-coefficients}
 C_H=\sum_{\{(j,k):\,U_jV_k=H\}}c_jd_k.
\end{equation}
This sum collects the coefficients of all ordered pairs that give
the same matrix product $H$.
Table \ref{tab:products} records the results of \eqref{eq:W-product}
for all sixteen pairs from \eqref{eq:UV}, together with $c_jd_k$.
\setlength{\LTcapwidth}{\textwidth}
\renewcommand{\arraystretch}{1.15}
\begin{longtable}{@{}c@{}}
\caption{Coefficients of $4U_jV_k$, products $U_jV_k$, and coefficients $c_jd_k$.}\label{tab:products}\\
\endfirsthead
\caption[]{Coefficients of $4U_jV_k$, products $U_jV_k$, and coefficients $c_jd_k$. (continued)}\\
\endhead
\multicolumn{1}{r}{\footnotesize Continued on the next page.}\\
\endfoot
\endlastfoot
\begin{tabular}{@{}l*{8}{c}@{}}
\toprule
$(j,k)$ & $(1,1)$ & $(1,2)$ & $(1,3)$ & $(1,4)$ & $(2,1)$ & $(2,2)$ & $(2,3)$ & $(2,4)$\\
\midrule
$A_0$ & 0 & 0 & 0 & 4 & 0 & 0 & 4 & 0\\
$A_1$ & 0 & 0 & -4 & 0 & 0 & 0 & 0 & 4\\
$A_2$ & -4 & 0 & 0 & 0 & 0 & 4 & 0 & 0\\
$A_3$ & 0 & -4 & 0 & 0 & -4 & 0 & 0 & 0\\
$U_jV_k$ & $-E_2$ & $-E_3$ & $-E_1$ & $I_2$ & $-E_3$ & $E_2$ & $I_2$ & $E_1$\\
$c_jd_k$ & $1$ & $-i$ & $-1$ & $i$ & $i$ & $1$ & $-i$ & $-1$\\
\bottomrule
\end{tabular}\\[1.5ex]
\begin{tabular}{@{}l*{8}{c}@{}}
\toprule
$(j,k)$ & $(3,1)$ & $(3,2)$ & $(3,3)$ & $(3,4)$ & $(4,1)$ & $(4,2)$ & $(4,3)$ & $(4,4)$\\
\midrule
$A_0$ & 0 & 4 & 0 & 0 & 4 & 0 & 0 & 0\\
$A_1$ & 4 & 0 & 0 & 0 & 0 & -4 & 0 & 0\\
$A_2$ & 0 & 0 & -4 & 0 & 0 & 0 & 0 & 4\\
$A_3$ & 0 & 0 & 0 & 4 & 0 & 0 & 4 & 0\\
$U_jV_k$ & $E_1$ & $I_2$ & $-E_2$ & $E_3$ & $I_2$ & $-E_1$ & $E_3$ & $E_2$\\
$c_jd_k$ & $1$ & $-i$ & $-1$ & $i$ & $i$ & $1$ & $-i$ & $-1$\\
\bottomrule
\end{tabular}\\
\end{longtable}
Each data column corresponds to the pair $(j,k)$ in its heading.
The table is arranged in two parts, each containing eight pairs.
The rows $A_0,A_1,A_2,A_3$ give the coefficients of $4U_jV_k$.
Dividing that matrix by $4$ gives the entry in the row $U_jV_k$,
and the last row gives $c_jd_k$.
For example, the column $(1,1)$ records $U_1V_1=-E_2$
and $c_1d_1=1$. Adding the scalar coefficients in columns
with the same product gives Table \ref{tab:cancellations}.

\setlength{\LTcapwidth}{\textwidth}
\renewcommand{\arraystretch}{1.2}
\begin{longtable}{@{}cll@{}}
\caption{Sums of coefficients for the distinct matrix products.}\label{tab:cancellations}\\
\toprule
$H$ & $C_H$ & Value\\
\midrule
\endfirsthead
\caption[]{Sums of coefficients for the distinct matrix products. (continued)}\\
\toprule
$H$ & $C_H$ & Value\\
\midrule
\endhead
\midrule
\multicolumn{3}{r}{\footnotesize Continued on the next page.}\\
\endfoot
\bottomrule
\endlastfoot
$I_2$ & $c_1d_4+c_2d_3+c_3d_2+c_4d_1$ & $i-i-i+i=0$\\
$-E_2$ & $c_1d_1+c_3d_3$ & $1-1=0$\\
$-E_3$ & $c_1d_2+c_2d_1$ & $-i+i=0$\\
$-E_1$ & $c_1d_3+c_4d_2$ & $-1+i(-i)=0$\\
$E_2$ & $c_2d_2+c_4d_4$ & $i(-i)+i^2=0$\\
$E_1$ & $c_2d_4+c_3d_1$ & $i^2+1=0$\\
$E_3$ & $c_3d_4+c_4d_3$ & $i-i=0$\\
\end{longtable}
In Table \ref{tab:cancellations}, the first column lists the
common matrix product $H$.
The second adds the coefficients of all columns in Table \ref{tab:products}
with that product, and the third evaluates the sum. For example,
the four pairs $(1,4)$, $(2,3)$, $(3,2)$ and $(4,1)$ have
product $I_2$, so $C_{I_2}=i-i-i+i=0$.
The tables give the cancellation needed in the operator calculation.
\begin{lemma}\label{lem:matrix-cancellation}
For the matrices and coefficients in \eqref{eq:UV}, with
$\mathcal H$ as in \eqref{E7},
$$
 U_jV_k\in\mathcal H\quad(1\leq j,k\leq4),\quad
 C_H=0\quad(H\in\mathcal H),
$$
where $C_H$ is defined in \eqref{eq:cancellation-coefficients}.
\end{lemma}

\begin{proof}
Substituting the choices of signs from \eqref{eq:UV} into
\eqref{eq:W-product} gives the sixteen cases in Table \ref{tab:products}.
All entries in its rows labelled $U_jV_k$ belong to $\mathcal H$.
Summing the corresponding coefficients according to
\eqref{eq:cancellation-coefficients} gives Table \ref{tab:cancellations},
whose last column is zero in every row, which completes the proof.
\end{proof}

The Gaussian representations in \eqref{eq:symbol-sums}
allow us to apply these cancellations to the operators.

\begin{proposition}\label{prop:operator-products}
For the functions $f,g$ in \eqref{eq:f} and \eqref{eq:g},
$$
 T_fT_g=0,\quad T_f\ne0,\quad T_g\ne0
 \quad\hbox{on }F^2(\Cn).
$$
\end{proposition}

\begin{proof}
The estimates in Lemma \ref{lem:symbol-properties} show that
$T_f$ and $T_g$ are bounded on $F^2(\Cn)$.
Fix $a\in\Cn$. By \eqref{eq:kernel-gaussian} and
\eqref{eq:symbol-sums},
\begin{equation}\label{eq:g-kernel-action}
 T_gK_a=\kappa_n\sum_{j=1}^4c_jK_{D_{U_j}^*a}.
\end{equation}
Applying $T_f$ to these kernel functions and using \eqref{eq:kernel-gaussian}
and \eqref{eq:symbol-sums} again gives
\begin{equation}\label{E8}
\begin{aligned}
 T_fT_gK_a
 &=\kappa_n\sum_{j=1}^4c_j T_fK_{D_{U_j}^*a}\\
 &=\kappa_n^2\sum_{j,k=1}^4c_jd_k
            K_{D_{V_k}^*D_{U_j}^*a}.
\end{aligned}
\end{equation}
Since $D_{V_k}^*D_{U_j}^*=(D_{U_j}D_{V_k})^*$,
\eqref{eq:D-definition} gives
\begin{align*}
 &K_{D_{V_k}^*D_{U_j}^*a}(z)
 =\exp\bigl((D_{U_j}D_{V_k}z)\cdot a\bigr)
\end{align*}
and 
\begin{equation*}
D_{U_j}D_{V_k}=\diag(U_jV_k/16,I_{n-2}/4).
\end{equation*}
To separate the dependence on $U_jV_k$, write
$$
 z=\begin{pmatrix}\xi\\z'\end{pmatrix},\quad
 a=\begin{pmatrix}\eta\\a'\end{pmatrix},
$$
where $\xi=(z_1,z_2)^T$ and $\eta=(a_1,a_2)^T\in\mathbb C^2$, while $z'=(z_3,\ldots,z_n)^T$ and
$a'=(a_3,\ldots,a_n)^T$ belong to $\mathbb C^{n-2}$.
Since $a^*=\begin{pmatrix}\eta^*&(a')^*\end{pmatrix}$, we have
$$
\begin{aligned}
 (D_{U_j}D_{V_k}z)\cdot a
 &=\begin{pmatrix}\eta^*&(a')^*\end{pmatrix}
   \begin{pmatrix}U_jV_k\xi/16\\z'/4\end{pmatrix} =\eta^*U_jV_k\xi/16+(a')^*z'/4\\
 &=\left(\frac{U_jV_k\xi}{16}\right)\cdot\eta
        +(z'/4)\cdot a'.
\end{aligned}
$$
Substituting this expression for the Hermitian inner product
into \eqref{E8} gives
\begin{equation}\label{eq:product}
 (T_fT_gK_a)(z)
 =\kappa_n^2 e^{(z'/4)\cdot a'}
       \sum_{j,k=1}^4c_jd_k
             \exp\left(\left(\frac{U_jV_k\xi}{16}\right)
                       \cdot\eta\right).
\end{equation}

Terms with the same matrix product in \eqref{eq:product}
have the same exponential factor. Collecting these terms according to
\eqref{eq:cancellation-coefficients} and using
Lemma \ref{lem:matrix-cancellation}, we obtain
$$
\begin{aligned}
 &\sum_{j,k=1}^4c_jd_k
     \exp\left(\left(\frac{U_jV_k\xi}{16}\right)\cdot\eta\right)\\
 &\quad=\sum_{H\in\mathcal H}
     \left(\sum_{\{(j,k):\,U_jV_k=H\}}c_jd_k\right)
     \exp\left(\left(\frac{H\xi}{16}\right)\cdot\eta\right)\\
 &\quad=\sum_{H\in\mathcal H}C_H
     \exp\left(\left(\frac{H\xi}{16}\right)\cdot\eta\right)=0.
\end{aligned}
$$
Equation \eqref{eq:product} therefore gives
$T_fT_gK_a=0$ for every $a\in\Cn$.
The kernels have dense linear span in $F^2(\Cn)$, and $T_fT_g$
is bounded, so $T_fT_g=0$ on $F^2(\Cn)$.

To show that $T_g$ is nonzero, we set $a=0$ in
\eqref{eq:g-kernel-action} and use \eqref{eq:UV} and $K_0=1$ to obtain
$$
 T_gK_0=\kappa_n(c_1+c_2+c_3+c_4)
       =2\kappa_n(1+i)\ne0.
$$
For $f$, the coefficient sum is $d_1+d_2+d_3+d_4=0$, so
applying $T_f$ to $K_0$ alone does not show that $T_f$ is nonzero.
We now apply $T_f$ to $K_{a_0}$ with
$a_0=(1,0,\ldots,0)^T\in\Cn$. Since only the first coordinate
of $a_0$ is nonzero, one can see 
$$
 (D_{V_k}a_0)\cdot a_0
 =a_0^*D_{V_k}a_0=(V_k)_{11}/4.
$$
By \eqref{eq:W-entries} and \eqref{eq:UV}, the $(1,1)$ entries of
$V_1,V_4$ are $(1+i)/2$, and those of $V_2,V_3$ are
$(1-i)/2$. For these two pairs, the coefficient sums are
$$
 d_1+d_4=1+i,\quad d_2+d_3=-i-1=-(1+i).
$$
Using \eqref{eq:D-definition}, \eqref{eq:kernel-gaussian} and
\eqref{eq:symbol-sums}, we obtain
$$
\begin{aligned}
 (T_fK_{a_0})(a_0)
 &=\kappa_n\sum_{k=1}^4d_k
                    \exp\bigl((D_{V_k}a_0)\cdot a_0\bigr)
                    \\
 &=\kappa_n\left\{(d_1+d_4)e^{(1+i)/8}
                      +(d_2+d_3)e^{(1-i)/8}\right\}\\
 &=\kappa_n(1+i)e^{1/8}(e^{i/8}-e^{-i/8})\\
 &=2i\kappa_n(1+i)e^{1/8}\sin(1/8)\ne0.
\end{aligned}
$$
Thus, $T_g\ne0$ and $T_f\ne0$ on $F^2(\Cn)$.
\end{proof}

\begin{proof}[Proof of Theorem \ref{thm:counterexample}]
Lemma \ref{lem:symbol-properties} gives the Schwartz regularity,
nonradiality and Gaussian bounds of $f,g$.
Proposition \ref{prop:operator-products} gives their zero Toeplitz
product and the nonvanishing of both operators. This proof is complete.
\end{proof}

\section{Proof of Theorem \ref{thm:moment-example}}\label{sec:moment-example}

Throughout this section, $n=1$. We use the definitions from the
introduction, so
$$
 \dd\mu(z)=\pi^{-1}e^{-|z|^2}\,\dd A(z),\quad
 \mathcal P=\mathbb C[z].
$$
For unbounded symbols, the domain of the product must be specified.
We use the extension in \cite[Definition 4.1]{BL}. It is clear that the functions
$e_j(z)=z^j/\sqrt{j!}$, $j\ge0$, form the standard orthonormal
basis of $F^2(\mathbb C)$.
For a measurable function $u:\mathbb C\to\mathbb C$, let
$\mathcal D(\widetilde T_u)$ consist of the functions $h\in F^2(\mathbb C)$
such that
\begin{equation}\label{E9}
 uh\overline{e}_j\in L^1(d\mu)\quad(j\ge0),\quad
 \sum_{j=0}^\infty
 \left|\int_{\mathbb C}uh\overline{e}_j\,\dd\mu\right|^2<\infty.
\end{equation}
Define $\widetilde T_u:\mathcal D(\widetilde T_u)\to F^2(\mathbb C)$ by
\begin{equation}\label{eq:moment-domain}
 \widetilde T_u h
 =\sum_{j=0}^\infty
 \left(\int_{\mathbb C}uh\overline{e}_j\,\dd\mu\right)e_j.
\end{equation}
By the discussion following Definition 4.1 in \cite{BL},
$T_u\subset\widetilde T_u$. In particular, the two operators agree
when $u$ is bounded.
The domain of $\widetilde T_f\widetilde T_g$ is
$$
 \{h\in\mathcal D(\widetilde T_g):
       \widetilde T_g h\in\mathcal D(\widetilde T_f)\}.
$$

We construct the symbols so that their extended Toeplitz product
vanishes on $\mathcal P$, as required by the formulation preceding
Question A in \cite[Section 6]{BL}. For this construction,
fix $c\in(1/2,1)$ and set
\begin{equation}\label{E10}
 \sigma=1-c\in(0,1/2),\quad q=2\sigma,\quad
 \gamma=q^{-1},\quad L=\sqrt{\frac{\pi}{2\sigma}}.
\end{equation}
Let $\mathbb Z$ denote the integers. We use the square lattice
$$
 \Gamma=\{jL+ikL:j,k\in\mathbb Z\}\subset\mathbb C.
$$
A function on $\mathbb C$ is periodic with respect to $\Gamma$
if its value is unchanged under both translations $z\mapsto z+L$
and $z\mapsto z+iL$.
Using the Jacobi theta function with parameter $i$
\cite[Section 20.2]{DLMF}, we define
$\vartheta_1,\psi:\mathbb C\to\mathbb C$ by
\begin{equation}\label{eq:moment-psi}
 \vartheta_1(\zeta)
 =-i\sum_{j\in\mathbb Z}(-1)^j
      e^{-\pi(j+1/2)^2}e^{i(2j+1)\zeta}\quad\text{and}\quad\psi(z)=e^{\sigma z^2}\vartheta_1(\pi z/L)
\end{equation}
for $\zeta,z\in \mathbb{C}$.
Define $v:\mathbb C\to[0,\infty)$ by
\begin{equation}\label{eq:moment-periodic-weight}
 v(w)=e^{-2\sigma|w|^2}|\psi(w)|^2.
\end{equation}
For $k=(k_1,k_2)^T\in\mathbb Z^2$, set
\begin{equation}\label{E18}
 \omega_k=(\omega_{k,1},\omega_{k,2})^T=\frac{2\pi}{L}k.
\end{equation}
Define the Fourier coefficients of $v$ by
\begin{equation}\label{E11}
 v_k=\frac1{L^2}\int_{[0,L]^2}
 v(x+iy)e^{-i(\omega_{k,1}x+\omega_{k,2}y)}\,\dd x\,\dd y.
\end{equation}
We write $v_0$ for $v_{(0,0)^T}$. For $H\in L^1(\mathbb C,dA)$,
we use the Fourier transform
\begin{equation}\label{E12}
 \widehat H(\nu)=\int_{\mathbb R^2}H(x+iy)
 e^{-i(x\nu_1+y\nu_2)}\,\dd x\,\dd y,
 \quad \nu=(\nu_1,\nu_2)^T\in\mathbb R^2.
\end{equation}
Choose $R>0$ and a nonnegative radial function
$\rho\in C_c^\infty(\mathbb R^2)$ such that
\begin{equation}\label{E13}
 R<\frac{\pi}{2L},\quad
 \rho=1\text{ near }0,\quad
 \operatorname{supp}\rho\subset\{\nu:|\nu|<R\}.
\end{equation}
Here $C_c^\infty(\mathbb R^2)$ is the space of smooth functions
with compact support. With the convention \eqref{E12}, define $r$
as the inverse Fourier transform of $\rho$ by
\begin{equation}\label{eq:moment-G}
 r(x+iy)=\frac1{(2\pi)^2}
  \int_{\mathbb R^2}\rho(\nu)e^{i(x\nu_1+y\nu_2)}\,\dd\nu.
\end{equation}
Here $\dd\nu=\dd\nu_1\,\dd\nu_2$. The two symbols are
\begin{equation}\label{eq:moment-symbols}
 f(w)=w e^{(1-2\sigma)|w|^2}\overline{\psi(w)}r(w)\quad\text{and} \quad
 g(w)=e^{-(\gamma-1)|w|^2}\psi(\gamma w).
\end{equation}

We first verify the properties of $\psi$ that control the growth
of the symbols and make the weight $v$ periodic.

\begin{lemma}\label{L1}
The function $\psi$ is entire, with $\psi(0)=\psi(L)=0$ and
$\psi(L/2)>0$. For every $z\in\mathbb C$, its translations satisfy

 \begin{align} \label{eq:moment-psi-periods}
 \psi(z+L)=-e^{2\sigma Lz+\sigma L^2}\psi(z),\quad
 \psi(z+iL)=-e^{-2i\sigma Lz+\sigma L^2}\psi(z).
 \end{align}
There is a constant $C>0$, depending only on $c$, such that
\begin{equation} \label{eq:moment-psi-bound}
 |\psi(z)|\le C e^{\sigma|z|^2}\quad(z\in\mathbb C).
\end{equation}
The function $v$ defined in \eqref{eq:moment-periodic-weight} is
smooth, nonnegative, nonzero, and periodic with respect to $\Gamma$. All its real partial derivatives are bounded.
\end{lemma}

\begin{proof}
Let $K\subset\mathbb C$ be compact and choose $M\ge0$ such that
$|\operatorname{Im}\zeta|\le M$ for all $\zeta\in K$. The absolute
value of the $j$-th term of the series defining $\vartheta_1$ in
\eqref{eq:moment-psi} is bounded on $K$ by
$$
 e^{-\pi(j+1/2)^2+|2j+1|M}.
$$
To show that these bounds are summable, put $t=|j+1/2|$ and note that
$$
\begin{aligned}
 -\pi t^2+2Mt
 &=-\frac{\pi}{2}t^2
   -\frac{\pi}{2}\left(t-\frac{2M}{\pi}\right)^2
   +\frac{2M^2}{\pi}\\
 &\le-\frac{\pi}{2}t^2+\frac{2M^2}{\pi}.
\end{aligned}
$$
It follows that
$$
 \sum_{j\in\mathbb Z}e^{-\pi(j+1/2)^2+|2j+1|M}
 \le e^{2M^2/\pi}
     \sum_{j\in\mathbb Z}e^{-\frac{\pi}{2}(j+1/2)^2}
 <\infty.
$$
The Weierstrass $M$-test \cite{DLMF} now gives
absolute and uniform convergence of the series for $\vartheta_1$
on $K$. Since $K$ is arbitrary and each term is entire, local
uniform convergence implies that $\vartheta_1$ is entire.
The definition in \eqref{eq:moment-psi} then shows that $\psi$ is
entire as the product of $e^{\sigma z^2}$ and
$\vartheta_1(\pi z/L)$.

By the absolute convergence of the series for $\vartheta_1$
in \eqref{eq:moment-psi}, we may pair the terms indexed by $j$
and $-j-1$ for $j=0,1,\ldots$ to obtain
$$
 \vartheta_1(0)
 =-i\sum_{j=0}^{\infty}
   \bigl[(-1)^j+(-1)^{-j-1}\bigr]e^{-\pi(j+1/2)^2}
 =0.
$$
The terms cancel because $(-j-1+1/2)^2=(j+1/2)^2$ and
$(-1)^{-j-1}=-(-1)^j$. By \eqref{eq:moment-psi}, this also gives
$\psi(0)=0$. At $z=L/2$, we have
$$
 \psi(L/2)
 =e^{\sigma L^2/4}
   \sum_{j\in\mathbb Z}e^{-\pi(j+1/2)^2}>0.
$$
To determine the translations of $\psi$, we use the series for
$\vartheta_1$. Since this series converges absolutely, substitution
and the change of index $j\mapsto j+1$ give
\begin{equation}\label{EE1}
   \vartheta_1(\zeta+\pi)=-\vartheta_1(\zeta)\quad\text{and}\quad
 \vartheta_1(\zeta+i\pi)=-e^{\pi-2i\zeta}\vartheta_1(\zeta).
\end{equation}
The definition of $\psi$ in \eqref{eq:moment-psi} and equation \eqref{EE1} give, for every $z\in\mathbb C$,
\begin{align*}
 \psi(z+L)
 &=e^{\sigma(z+L)^2}
   \vartheta_1\left(\frac{\pi z}{L}+\pi\right)\\
 &=-e^{2\sigma Lz+\sigma L^2}\psi(z).
\end{align*}
For the shift by $iL$, the definition of $L$ in \eqref{E10}
gives $\sigma L^2=\pi/2$ and $\pi/L=2\sigma L$, so we have
\begin{equation*}
 2i\sigma Lz-\frac{2i\pi z}{L}
 =2i\sigma Lz-4i\sigma Lz=-2i\sigma Lz
\end{equation*}
and
\begin{equation*}
  \pi-\sigma L^2=2\sigma L^2-\sigma L^2=\sigma L^2.
\end{equation*}
Applying the above identity for $\vartheta_1$ and using these
equalities, we obtain
$$
\begin{aligned}
 \psi(z+iL)
 &=e^{\sigma(z+iL)^2}
   \vartheta_1\left(\frac{\pi z}{L}+i\pi\right)\\
 &=-e^{\sigma z^2+2i\sigma Lz-\sigma L^2}
   e^{\pi-2i\pi z/L}
   \vartheta_1\left(\frac{\pi z}{L}\right)\\
 &=-e^{2i\sigma Lz-2i\pi z/L+\pi-\sigma L^2}\psi(z)\\
 &=-e^{-2i\sigma Lz+\sigma L^2}\psi(z).
\end{aligned}
$$
This proves \eqref{eq:moment-psi-periods}. Setting $z=0$ in
the first identity and using $\psi(0)=0$ also gives $\psi(L)=0$.

To obtain a bound for $\psi$, write $z=x+iy$, where
$x,y\in\mathbb R$. Taking absolute values in the first identity
in \eqref{eq:moment-psi-periods} gives
$$|\psi(z+L)|=e^{2\sigma Lx+\sigma L^2}|\psi(z)|.$$
It follows that 
$$
\begin{aligned}
 e^{-\sigma|z+L|^2}|\psi(z+L)|
 &=e^{-\sigma(|z|^2+2Lx+L^2)}
   e^{2\sigma Lx+\sigma L^2}|\psi(z)|\\
 &=e^{-\sigma|z|^2}|\psi(z)|.
\end{aligned}
$$
The same calculation with the second identity in
\eqref{eq:moment-psi-periods} gives
$$
 e^{-\sigma|z+iL|^2}|\psi(z+iL)|
 =e^{-\sigma|z|^2}|\psi(z)|.
$$
The two translation identities show that
$e^{-\sigma|z|^2}|\psi(z)|$ is periodic with respect to $\Gamma$.
Since this function is continuous, it attains a finite maximum
$C>0$ on the closed square $Q=\{x+iy:0\le x,y\le L\}$. For any $z=x+iy\in\mathbb C$,
choose $j,k\in\mathbb Z$ such that $jL\le x<(j+1)L$ and
$kL\le y<(k+1)L$. Then $z_0=z-jL-ikL\in Q$, which shows that
the translates of $Q$ by elements of $\Gamma$ cover $\mathbb C$.
Periodicity now gives
$$
 e^{-\sigma|z|^2}|\psi(z)|
 =e^{-\sigma|z_0|^2}|\psi(z_0)|\le C,
$$
which proves \eqref{eq:moment-psi-bound}.
By \eqref{eq:moment-periodic-weight},
$v(z)=\bigl(e^{-\sigma|z|^2}|\psi(z)|\bigr)^2$, so $v$ is also
periodic with respect to $\Gamma$. The function $v$ is smooth
because $\psi$ is entire, and it is nonzero because $\psi(L/2)>0$.
Every real partial derivative of $v$ is continuous and periodic
with respect to $\Gamma$, so its boundedness on $Q$ gives a bound
on $\mathbb C$ by the same translation argument. The proof is complete.
\end{proof}

We will use that multiplication by a polynomial in the real
coordinates preserves the Schwartz class and that Schwartz
functions are integrable, see
\cite[Chapter VIII, Section 4]{Knapp2016}.
In particular, for every $h\in\mathcal S(\mathbb C)$ and
$N\in\mathbb N_0$,
\begin{equation}\label{E19}
 \int_{\mathbb C}(1+|w|^2)^N|h(w)|\,\dd A(w)
 <\infty.
\end{equation}
Taking $N=0$ gives $h\in L^1(dA)$, and larger values of $N$
show that every polynomial in $w$ and $\overline{w}$ times $h$ is
integrable.

Using the smoothness and periodicity of $v$, we combine its
Fourier series with the support condition \eqref{E13} on $\rho$
to obtain the integral identities needed for the operator product.
\begin{proposition}\label{L2}
The functions $r$ and $vr$ belong to $\mathcal S(\mathbb C)$, and
$r(w)>0$ for $|w|\le L$. The Fourier transform of $vr$ satisfies
\begin{equation}\label{eq:moment-flat}
 \widehat{vr}=v_0>0\quad\hbox{on a neighborhood of }0.
\end{equation}
The following integral identities hold
\begin{gather} \label{eq:moment-identities}
 \int_{\mathbb C}v(w)r(w)\,\dd A(w)=v_0,\quad
 \int_{\mathbb C}v(w)r(w)\overline{w}^j\,\dd A(w)=0\quad(j\ge1),
\end{gather}
 and
\begin{equation}\label{E14}
\int_{\mathbb C}wv(w)r(w)w^m\overline{w}^j\,\dd A(w)=0
 \quad(m,j\ge0).
\end{equation}
The Fourier transform of $w\mapsto wv(w)r(w)$ vanishes near $0$.
\end{proposition}

\begin{proof}
Recall that $\rho\in C_c^\infty(\mathbb R^2)$ is the nonnegative
radial function chosen in \eqref{E13}, and that \eqref{eq:moment-G}
defines $r$ as its inverse Fourier transform. Since
$C_c^\infty(\mathbb R^2)\subset\mathcal S(\mathbb R^2)$,
Fourier inversion \cite[Proposition 8.10]{Knapp2016},
with the normalization \eqref{E12}, gives
\begin{equation}\label{E20}
 r\in\mathcal S(\mathbb C)\quad \text{and} \quad\widehat r=\rho.
\end{equation}
The bounded derivatives of $v$ from Lemma \ref{L1}, together
with the product rule, then give $vr\in\mathcal S(\mathbb C)$.

Write $w=x+iy$, where $x,y\in\mathbb R$. The function $\rho$
chosen in \eqref{E13} is nonnegative and radial, so it is real
and satisfies $\rho(-\nu)=\rho(\nu)$. Expanding the exponential
in \eqref{eq:moment-G} gives
$$
\begin{aligned}
 r(w)
 &=\frac1{(2\pi)^2}\int_{\mathbb R^2}
   \rho(\nu)\cos(x\nu_1+y\nu_2)\,\dd\nu\\
 &\quad+\frac{i}{(2\pi)^2}\int_{\mathbb R^2}
   \rho(\nu)\sin(x\nu_1+y\nu_2)\,\dd\nu.
\end{aligned}
$$
Both integrals converge absolutely because $\rho$ is smooth and
compactly supported. Using $\rho(-\nu)=\rho(\nu)$ and the oddness
of sine, the second integral equals 0.
For $|w|\le L$ and
$\nu\in\operatorname{supp}\rho$,  it follows from  the Cauchy–Schwarz Inequality and \eqref{E13} that 
$$
 |x\nu_1+y\nu_2|
 \le\sqrt{x^2+y^2}\sqrt{\nu_1^2+\nu_2^2}
 =|w||\nu|\le LR<\frac\pi2.
$$
Hence $\cos(x\nu_1+y\nu_2)\ge\cos(LR)>0$ on the support
of $\rho$. Since $\rho=1$ near $0$, choose
$\varepsilon\in(0,R)$ such that $\rho(\nu)=1$ for
$|\nu|\le\varepsilon$. The nonnegativity of $\rho$ gives
$$
 \int_{\mathbb R^2}\rho(\nu)\,\dd\nu
 \ge\int_{|\nu|\le\varepsilon}1\,\dd\nu
 =\pi\varepsilon^2.
$$
Combining these bounds in the cosine integral, we obtain
$$
 r(w)
 \ge\frac{\cos(LR)}{(2\pi)^2}
      \int_{\mathbb R^2}\rho(\nu)\,\dd\nu
 \ge\frac{\varepsilon^2\cos(LR)}{4\pi}>0
 \quad(|w|\le L).
$$
This shows that the remaining cosine integral satisfies
$$
 r(w)=\frac1{(2\pi)^2}\int_{\mathbb R^2}
 \rho(\nu)\cos(x\nu_1+y\nu_2)\,\dd\nu>0
 \quad(|w|\le L).
$$

We next estimate the Fourier coefficients in \eqref{E11}
by integration by parts. For each fixed $y\in\mathbb R$,
Lemma \ref{L1} gives $v(x+L+iy)=v(x+iy)$ for all
$x\in\mathbb R$. For every integer $j\ge0$, differentiating
this identity $j$ times with respect to $x$ and setting $x=0$
gives $$\partial_x^jv(L+iy)=\partial_x^jv(iy).$$
By \eqref{E18}, $e^{-i\omega_{k,1}L}=e^{-2\pi i k_1}=1$, the boundary term satisfies
\begin{align*}
 0=\left[\partial_x^jv(x+iy)e^{-i\omega_{k,1}x}\right]_{x=0}^{x=L}
 &=\partial_x^jv(L+iy)e^{-2\pi i k_1}
   -\partial_x^jv(iy).
\end{align*}
The same cancellation holds at $y=0,L$ when integrating in $y$.
So, every boundary term vanishes when we integrate by parts.
For $k_1\ne0$, an antiderivative of $e^{-i\omega_{k,1}x}$ is
$-(i\omega_{k,1})^{-1}e^{-i\omega_{k,1}x}$.
For each fixed $y\in\mathbb R$ and each integer $j\ge0$,
integration by parts gives
$$
\begin{aligned}
 \int_0^L\partial_x^jv(x+iy)e^{-i\omega_{k,1}x}\,\dd x
 &=-\frac1{i\omega_{k,1}}
 \left[\partial_x^jv(x+iy)e^{-i\omega_{k,1}x}\right]_{x=0}^{x=L}\\
 &\quad+\frac1{i\omega_{k,1}}
 \int_0^L\partial_x^{j+1}v(x+iy)e^{-i\omega_{k,1}x}\,\dd x\\
 &=\frac1{i\omega_{k,1}}
 \int_0^L\partial_x^{j+1}v(x+iy)e^{-i\omega_{k,1}x}\,\dd x,
\end{aligned}
$$
where the last equality uses the boundary cancellation above.
The case $j=0$ gives the first integration by parts.

For every integer $N\ge1$, we apply this identity successively for
$j=0,\ldots,N-1$. Each application adds one derivative of $v$
and a factor $(i\omega_{k,1})^{-1}$. Multiplying the resulting
identity by $L^{-2}e^{-i\omega_{k,2}y}$ and integrating in $y$
over $[0,L]$, we obtain from \eqref{E11} that
$$
 v_k=\frac1{L^2(i\omega_{k,1})^N}
 \int_{[0,L]^2}\partial_x^Nv(x+iy)
 e^{-i(\omega_{k,1}x+\omega_{k,2}y)}\,\dd x\,\dd y.
$$
For $N=0$, this formula is the definition \eqref{E11}.
Taking absolute values yields
$$
 |v_k|\le
 \left(\frac{L}{2\pi|k_1|}\right)^N
 \sup_{0\le x,y\le L}|\partial_x^Nv(x+iy)|.
$$
There is an analogous estimate with $k_2$ and $\partial_y^Nv$
when $k_2\ne0$. To obtain an estimate in terms of $|k|$, for
each $k\ne0$ we choose $\ell\in\{1,2\}$ such that
$|k_\ell|=\max\{|k_1|,|k_2|\}$. Since
$$
 |k|^2=|k_1|^2+|k_2|^2\le2|k_\ell|^2,
$$
we have $|k_\ell|\ge |k|/\sqrt2>0$.
Applying the corresponding integration by parts estimate and
using the bounded derivatives from Lemma \ref{L1} gives
$$
 |v_k|\le A_N|k|^{-N}\quad(k\ne0),
$$
where $A_N>0$ depends only on $N$ and $c$ and is independent of $k$.
After increasing the constant and including $k=0$, we obtain
$$
 |v_k|\le C_N(1+|k|)^{-N}
 \quad(k\in\mathbb Z^2, N\in\mathbb N_0),
$$
where $C_N>0$ depends only on $N$ and $c$ and is independent of $k$.

Taking $N=4$ gives absolute summability. The lattice estimate
used in \cite{Grafakos2014}                                                  
applies in real dimension two. Together with
$(1+|k|^2)^{-2}\le4(1+|k|)^{-4}$, it gives
\begin{equation}\label{E21}
 \sum_{k\in\mathbb Z^2}|v_k|
 \le C_4\sum_{k\in\mathbb Z^2}(1+|k|)^{-4}
 \le C_4\sum_{k\in\mathbb Z^2}(1+|k|^2)^{-2}<\infty.
\end{equation}
Since each exponential has absolute value one, the Fourier
series converges absolutely and uniformly. To identify its sum,
we make the change of variables $x=Lt_1/(2\pi)$ and
$y=Lt_2/(2\pi)$ in \eqref{E11}. This gives
$$
 v_k=\frac1{(2\pi)^2}\int_{[0,2\pi]^2}
 v\left(\frac{L}{2\pi}(t_1+it_2)\right)
 e^{-i(k_1t_1+k_2t_2)}\,\dd t_1\,\dd t_2.
$$
Thus $v_k$ is the $k$-th Fourier coefficient of the smooth
function $(t_1,t_2)\mapsto v(L(t_1+it_2)/(2\pi))$, which is
$2\pi$ periodic in each real coordinate. By \eqref{E21}, Fourier inversion
\cite[Proposition 3.2.5]{Grafakos2014} applies after rescaling
to period $2\pi$. The Fourier series then sums to this function almost everywhere. Both functions are continuous, so the equality
holds everywhere and gives
\begin{equation}\label{E15}
 v(x+iy)=\sum_{k\in\mathbb Z^2}v_k
 e^{i(\omega_{k,1}x+\omega_{k,2}y)}.
\end{equation}

For $k=(0,0)^T$, the definition \eqref{E11} gives the average
of $v$ over $[0,L]^2$,
$$
 v_0=\frac1{L^2}\int_{[0,L]^2}v(x+iy)\,\dd x\,\dd y.
$$
By Lemma \ref{L1}, $v$ is continuous, nonnegative, nonzero and
periodic with respect to $\Gamma$. Set
$Q=\{x+iy:0\le x,y\le L\}$. There is a point $w_0\in Q^\circ$
such that $v(w_0)>0$. Otherwise, nonnegativity would give
$v=0$ on $Q^\circ$, and continuity would give $v=0$ on
$Q=\overline{Q^\circ}$. Periodicity would then imply $v\equiv0$
on $\mathbb C$, which is a contradiction.
Put $a=v(w_0)>0$. By continuity at $w_0$ and the fact that
$w_0\in Q^\circ$, we can choose $\varepsilon>0$ such that
$$
 D=\{w\in\mathbb C:|w-w_0|<\varepsilon\}\subset Q^\circ \quad \text{and} \quad
 \quad |v(w)-a|<\frac a2\quad(w\in D).
$$
So we have $v(w)\ge a/2$ on $D$. Since $v\ge0$ on $Q$, we obtain
$$
\begin{aligned}
 v_0
 &=\frac1{L^2}\int_Q v(w)\,\dd A(w)
 \ge\frac1{L^2}\int_D v(w)\,\dd A(w)\\
 &\ge\frac{a}{2L^2}\int_D1\,\dd A(w)
 =\frac{a\pi\varepsilon^2}{2L^2}>0.
\end{aligned}
$$

We now insert \eqref{E15} into the Fourier transform of $vr$.
By \eqref{E20}, we may apply \eqref{E19} with $h=r$ and $N=0$
to obtain $\int_{\mathbb C}|r(w)|\,\dd A(w)<\infty$.
Combining this with \eqref{E21}, for every
$\nu=(\nu_1,\nu_2)^T\in\mathbb R^2$ we obtain
$$
\begin{aligned}
 &\qquad\sum_{k\in\mathbb Z^2}\int_{\mathbb R^2}
 \left|v_kr(x+iy)e^{i(\omega_{k,1}x+\omega_{k,2}y)}
 e^{-i(x\nu_1+y\nu_2)}\right|\,\dd x\,\dd y\\
 &\quad=\left(\sum_{k\in\mathbb Z^2}|v_k|\right)
 \int_{\mathbb C}|r(w)|\,\dd A(w)<\infty.
\end{aligned}
$$
We may therefore interchange the sum and the integral.
Using the convention \eqref{E12}, the Fourier transform
identity in \eqref{E20} and \eqref{E15}, we obtain
\begin{equation}\label{E16}
\begin{aligned}
 \widehat{vr}(\nu)
 &=\int_{\mathbb R^2}v(x+iy)r(x+iy)
   e^{-i(x\nu_1+y\nu_2)}\,\dd x\,\dd y\\
 &=\sum_{k\in\mathbb Z^2}v_k\int_{\mathbb R^2}r(x+iy)
   e^{-i[x(\nu_1-\omega_{k,1})+y(\nu_2-\omega_{k,2})]}
   \,\dd x\,\dd y\\
 &=\sum_{k\in\mathbb Z^2}v_k\widehat r(\nu-\omega_k)
  =\sum_{k\in\mathbb Z^2}v_k\rho(\nu-\omega_k).
\end{aligned}
\end{equation}

By \eqref{E13}, choose $\delta\in(0,R)$ such that
$\rho(\nu)=1$ for $|\nu|<\delta$. For $k\ne0$, \eqref{E18} gives
$$
 |\omega_k|=\frac{2\pi}{L}|k|\ge\frac{2\pi}{L}.
$$
For $|\nu|<\delta$, the triangle inequality and \eqref{E13} yield
$$
 |\nu-\omega_k|
 \ge|\omega_k|-|\nu|
 >\frac{2\pi}{L}-R>R.
$$
The support condition in \eqref{E13} implies 
$\rho(\nu-\omega_k)=0$ for every $k\ne0$, so \eqref{E16} reduces to
$$
 \widehat{vr}(\nu)=v_0\rho(\nu)=v_0
 \quad(|\nu|<\delta),
$$
which proves \eqref{eq:moment-flat}.

To prove \eqref{eq:moment-identities} and \eqref{E14}, fix
$m,j\in\mathbb N_0$ and apply \eqref{E19} with $h=vr$ and
$N=m+j$ to obtain
$$
 \int_{\mathbb C}|v(w)r(w)|\,|w|^{m+j}\,\dd A(w)
 \le\int_{\mathbb C}(1+|w|^2)^{m+j}|v(w)r(w)|\,\dd A(w)
 <\infty.
$$
Write $w=x+iy$ with $x,y\in\mathbb R$, and let
$\nu=(\nu_1,\nu_2)^T\in\mathbb R^2$ be the Fourier variable.
The bound above allows repeated differentiation with respect
to $\nu_1,\nu_2$ under the integral in \eqref{E12}, giving
\begin{align}\label{E23}
 &\qquad(\partial_{\nu_1}+i\partial_{\nu_2})^m
   (\partial_{\nu_1}-i\partial_{\nu_2})^j\widehat{vr}(\nu)\nonumber\\
 &\quad=(-i)^{m+j}\int_{\mathbb C}
 v(w)r(w)w^m\overline{w}^j e^{-i(x\nu_1+y\nu_2)}\,\dd A(w)
\end{align}
for every $\nu\in\mathbb R^2$. By \eqref{eq:moment-flat},
$\widehat{vr}=v_0$ on a neighborhood of $0$,  we can see all derivatives
of total order $m+j\ge1$ vanish at $0$. Evaluating \eqref{E23}
at $\nu=0$ therefore gives
$$
\begin{aligned}
 &\qquad\int_{\mathbb C}v(w)r(w)w^m\overline{w}^j\,\dd A(w)\\
 &\quad=i^{m+j}
 \left[(\partial_{\nu_1}+i\partial_{\nu_2})^m
       (\partial_{\nu_1}-i\partial_{\nu_2})^j
       \widehat{vr}(\nu)\right]_{\nu=0}\\
 &\quad=
 \begin{cases}
 v_0,&m=j=0,\\
 0,&m+j\ge1.
 \end{cases}
\end{aligned}
$$
Taking $m=0$ gives \eqref{eq:moment-identities}. Replacing $m$
by $m+1$ gives
$$
 \int_{\mathbb C}wv(w)r(w)w^m\overline{w}^j\,\dd A(w)
 =\int_{\mathbb C}v(w)r(w)w^{m+1}\overline{w}^j\,\dd A(w)
 =0\quad
$$
for $(m,j\in\mathbb N_0)$, which proves \eqref{E14}. Finally, \eqref{E23} with $(m,j)=(1,0)$
and \eqref{eq:moment-flat} give
$$
 \widehat{wvr}(\nu)
 =i(\partial_{\nu_1}+i\partial_{\nu_2})\widehat{vr}(\nu)
 =0\quad(|\nu|<\delta).
$$ This completes the proof.
\end{proof}

We now use the bounds for $\psi$ and $r$ to control the growth
of $f$ and $g$. The zeros of $\psi$ will also give nonradiality
of both symbols.

\begin{lemma}\label{L3}
The symbols in \eqref{eq:moment-symbols} are nonzero smooth
functions in $\mathcal D_c$. Neither is radial, even up to equality
almost everywhere. More precisely, with $\beta=1-1/(4\sigma)<1/2$,
there is a constant $C>0$, depending only on $c$ and $\rho$, such that
\begin{equation}\label{E1}
 |f(w)|\le C e^{c|w|^2},\quad
 |g(w)|\le C e^{\beta|w|^2}\quad(w\in\mathbb C).
\end{equation}
The symbol $g$ is bounded when $c\ge3/4$.
\end{lemma}

\begin{proof}
The the definition \eqref{eq:moment-symbols} and Lemma \ref{L1}
and Proposition \ref{L2} give smoothness of $f$ and $g$. By Proposition \ref{L2} again, $r\in\mathcal S(\mathbb C)$,  $wr(w)$
is bounded. Combining this bound with \eqref{eq:moment-psi-bound},
we obtain
$$
 |f(w)|\le C e^{(1-\sigma)|w|^2}
$$
for some constant $C>0$.
For $g$, the bound for $\psi$ and the parameter relations in
\eqref{E10} give
$$
 |g(w)|
 \le C\exp\left[
   \left(1-\gamma+\sigma\gamma^2\right)|w|^2\right]
 =C\exp\left[\left(1-\frac1{4\sigma}\right)|w|^2\right].
$$
The inequalities
$$
 1-\frac1{4\sigma}<\frac12,
 \quad
 1-\frac1{4\sigma}\le1-\sigma=c
$$
and the preceding bounds show that $f,g\in\mathcal D_c$ by
\eqref{E5}. The bound for $g$ also shows that it is bounded
when $\sigma\le1/4$, or equivalently $c\ge3/4$.

For nonradiality, Lemma \ref{L1} gives $\psi(L)=0$.
Because $\psi$ is a nonzero entire function, it cannot vanish on
the whole circle $|z|=L$. Choose $\zeta$ on this circle with
$\psi(\zeta)\ne0$. Proposition \ref{L2} gives $r(\zeta)>0$, then
\eqref{eq:moment-symbols} yields
$$
 f(L)=0,\quad f(\zeta)\ne0,\quad
 g(L/\gamma)=0,\quad g(\zeta/\gamma)\ne0.
$$
Both symbols are continuous, and the two points in each pair
have equal modulus but different symbol values. As in the proof
of Lemma \ref{lem:symbol-properties}, these values show that
neither symbol is radial, even up to equality almost everywhere.
The proof is complete.
\end{proof}

To compute the Toeplitz actions of the symbols in
\eqref{eq:moment-symbols}, we use the following extension of
\eqref{eq:scalar} for positive real $\lambda$. The formula follows
from the Gaussian integral identity in \cite[(2.1)]{Bauer2015}.
Let $\lambda>0$,
$0\le\tau<\lambda$ and $\ell\ge0$, and let
$\varphi:\mathbb C\to\mathbb C$ be entire with
$|\varphi(w)|\le C e^{\tau|w|^2+\ell|w|}$ for every $w\in\mathbb C$
and some constant $C>0$. Then
\begin{equation}\label{E2}
 \frac1\pi\int_{\mathbb C}\varphi(w)
 e^{-\lambda|w|^2+a\overline{w}}\,\dd A(w)
 =\lambda^{-1}\varphi(a/\lambda)\quad(a\in\mathbb C).
\end{equation}
The integral converges absolutely, locally uniformly in $a$.

Equation \eqref{E2} gives the operator actions needed to verify
the polynomial domains and to find test functions that show
both operators are nonzero.

\begin{lemma}\label{P1}
For $f,g$ defined in \eqref{eq:moment-symbols}, the extended
Toeplitz operators $\widetilde T_f$ and $\widetilde T_g$ in
$F^2(\mathbb C)$ are nonzero, and their domains contain
$\mathcal P=\mathbb C[z]$.
\end{lemma}

\begin{proof}
To compute $\widetilde T_g p$ for a fixed $p\in\mathcal P$, we use the
function $\psi$ defined in \eqref{eq:moment-psi}, which is entire
by Lemma \ref{L1}. The estimate for $g$ in \eqref{E1} gives
$g\in\mathcal D_d$ for some $0<d<1/2$, so
$\mathcal P\subset\mathcal D(T_g)\subset\mathcal D(\widetilde T_g)$.
For each fixed $z\in\mathbb C$, we integrate in
$w=x+iy\in\mathbb C$, where $x,y\in\mathbb R$, and apply
\eqref{E2} with $\varphi(w)=\psi(\gamma w)p(w)$,
$\lambda=\gamma$ and $a=z$.
For the parameters $\sigma,q,\gamma$ in \eqref{E10},
$\sigma\gamma^2=\gamma/2<\gamma$. The bound
\eqref{eq:moment-psi-bound} therefore verifies the growth
condition in \eqref{E2} and gives
\begin{equation}\label{eq:moment-right-polynomial}
 \begin{aligned}
 \widetilde T_g p(z)
 &=\frac1\pi\int_{\mathbb C}
       \psi(\gamma w)p(w)e^{z\overline{w}-\gamma|w|^2}\,\dd A(w)\\
 &=\gamma^{-1}\psi(z)p(z/\gamma).
 \end{aligned}
\end{equation}

To verify the polynomial domain for $\widetilde T_f$, define
$\psi^\#:\mathbb C\to\mathbb C$ by
$$\psi^\#(\zeta)=\overline{\psi(\overline{\zeta})}.$$
Since $\psi$ is entire by Lemma \ref{L1}, so is $\psi^\#$.
In the following Gaussian integrals, $w\in\mathbb C$ is the
integration variable, and $s,t\in\mathbb C$ are independent
parameters. The parameter $s$ may vary freely in $\mathbb C$
and is independent of the evaluation point $z$ used below.
With $q=2\sigma>0$ fixed as in \eqref{E10}, fix
$m\in\mathbb N_0$ and a compact set $K\subset\mathbb C^2$
of parameter pairs $(s,t)$.
Choose $M>0$, depending only on $K$, such that $|s|+|t|\le M$
for $(s,t)\in K$. Let $C_0>0$ be the constant in
\eqref{eq:moment-psi-bound}. Since $q=2\sigma$ and
$M|w|\le\sigma|w|^2/2+M^2/(2\sigma)$, we have
\begin{equation}\label{E22}
 \begin{aligned}
|w|^j\left|e^{-q|w|^2+sw+t\overline{w}}
                  \overline{\psi(w)}\right|&\le C_0(1+|w|)^{m+1}e^{-\sigma|w|^2+M|w|}\\
 &\quad\le C(1+|w|)^{m+1}e^{-\sigma|w|^2/2}
 \end{aligned}
\end{equation}
for $(s,t)\in K$, $w\in\mathbb C$ and $j=0,\ldots,m+1$,
where $C=C_0e^{M^2/(2\sigma)}>0$.
The last bound is integrable with respect to $w$ and uniform
for $(s,t)\in K$. With $t$ fixed, each derivative in $s$
introduces a factor $w$, so \eqref{E22} permits differentiation
under the integral in $w$ up to order $m+1$.

We now apply \eqref{E2} with $\lambda=q$,
$\varphi(w)=e^{\overline{t}w}\psi(w)$ and $a=\overline{s}$.
The growth condition holds because
$|\varphi(w)|\le C_0e^{\sigma|w|^2+|t||w|}$ and $\sigma<q$.
Then we see 
$$
 \frac1\pi\int_{\mathbb C}
 e^{-q|w|^2+\overline{t}w+\overline{s}\,\overline{w}}
 \psi(w)\,\dd A(w)
 =q^{-1}e^{\overline{s}\,\overline{t}/q}
          \psi(\overline{s}/q).
$$
Since $q>0$ is real, taking complex conjugates yields
 \begin{align*}
 &\frac1\pi\int_{\mathbb C}
 e^{-q|w|^2+sw+t\overline{w}}\overline{\psi(w)}\,\dd A(w)=q^{-1}e^{st/q}\overline{\psi(\overline{s}/q)}
 =q^{-1}e^{st/q}\psi^\#(s/q).
 \end{align*}
For each fixed $t\in\mathbb C$, regard both sides as functions
of $s$. By \eqref{E22}, we may differentiate $m+1$ times in $s$
under the integral in $w$, with $m,q,t$ fixed. The product rule gives
\begin{align}\label{eq:moment-left-gaussian}
 &\frac1\pi\int_{\mathbb C}
 w^{m+1}e^{-q|w|^2+sw+t\overline{w}}
 \overline{\psi(w)}\,\dd A(w)\nonumber\\
 &\quad=q^{-1}\partial_s^{m+1}
              \big[e^{st/q}\psi^\#(s/q)\big]\nonumber\\
 &\quad=q^{-m-2}e^{st/q}
       \sum_{\ell=0}^{m+1}\binom{m+1}{\ell}
       t^{m+1-\ell}(\psi^\#)^{(\ell)}(s/q).
 \end{align}
Here $(\psi^\#)^{(\ell)}$ denotes the $\ell$th complex derivative
of $\psi^\#$.

We now compute the action on $w^m$, keeping $m$ fixed.
For this calculation, we use \eqref{eq:moment-G} for $r$,
with $\rho$ as in \eqref{E13}.
For each fixed $z\in\mathbb C$, the representation of $r$
introduces the integration variable
$\nu=(\nu_1,\nu_2)^T\in\mathbb R^2$. We write
$$
 b=b(\nu)=\frac{\nu_2+i\nu_1}{2},\quad
 t=t(z,\nu)=z-\overline{b(\nu)}.
$$
Thus $b$ depends only on $\nu$, and $t$ depends on $z$ and $\nu$.
For $w=x+iy$, these definitions give
$i(x\nu_1+y\nu_2)=bw-\overline{b}\overline{w}$.
By the definition of $f$ in \eqref{eq:moment-symbols} and
$q=2\sigma$ in \eqref{E10},
$$
 f(w)e^{-|w|^2}
 =w e^{-q|w|^2}\overline{\psi(w)}r(w).
$$
We insert \eqref{eq:moment-G} and use the preceding substitutions
$b=(\nu_2+i\nu_1)/2$ and $t=z-\overline b$ to write
$$
 -q|w|^2+z\overline{w}+i(x\nu_1+y\nu_2)
 =-q|w|^2+bw+t\overline{w}.
$$
For fixed $z$, the pair $(b,t)$ ranges over a compact set as
$\nu$ varies over $\operatorname{supp}\rho$.
Estimate \eqref{E22} with $j=m+1$ and the compact support of
$\rho$ therefore justify Fubini's theorem in the calculation below.
Using these facts, we obtain
a function $k_{f,m}:\mathbb C\to\mathbb C$ satisfying
\begin{align}\label{E3}
 k_{f,m}(z)
 &:=\int_{\mathbb C}f(w)w^m e^{z\overline{w}}\,\dd\mu(w)\nonumber\\
 &=\frac1\pi\int_{\mathbb C}
 w^{m+1}e^{-q|w|^2+z\overline{w}}
 \overline{\psi(w)}r(w)\,\dd A(w)\nonumber\\
 &=\frac1{\pi(2\pi)^2}\int_{\mathbb R^2}\rho(\nu)
 \left[\int_{\mathbb C}w^{m+1}e^{-q|w|^2+bw+t\overline{w}}
 \overline{\psi(w)}\,\dd A(w)\right]\dd\nu\nonumber\\
 &=\frac1{q(2\pi)^2}\int_{\mathbb R^2}\rho(\nu)
   \left.\partial_s^{m+1}
      \big[e^{st/q}\psi^\#(s/q)\big]\right|_{s=b}\,\dd\nu.
\end{align}
%
The last equality in \eqref{E3} follows from
\eqref{eq:moment-left-gaussian} with $s=b(\nu)$ and
$t=z-\overline{b(\nu)}$. To estimate $k_{f,m}(z)$, we keep $z\in\mathbb C$ fixed
and estimate the integrand in \eqref{E3}. It follows that 
\begin{equation}\label{E24}
\begin{aligned}
 k_{f,m}(z)
 &=\frac{1}{q^{m+2}(2\pi)^2}
   \int_{\mathbb R^2}\rho(\nu)
   e^{b(\nu)(z-\overline{b(\nu)})/q}\\
 &\quad\times\sum_{\ell=0}^{m+1}\binom{m+1}{\ell}
   \bigl(z-\overline{b(\nu)}\bigr)^{m+1-\ell}
   (\psi^\#)^{(\ell)}\bigl(b(\nu)/q\bigr)\,\dd\nu.
\end{aligned}
\end{equation}
By \eqref{E13},
$$
 |b(\nu)|=\frac{|\nu|}{2}\le\frac R2,
 \quad
 \left|\frac{b(\nu)}q\right|\le\frac{R}{2q}
 \quad(\nu\in\operatorname{supp}\rho).
$$
Since $\psi^\#$ is entire, we may define
$$
 M_m:=1+\max_{0\le\ell\le m+1}
 \sup_{|\zeta|\le R/(2q)}
 \bigl|(\psi^\#)^{(\ell)}(\zeta)\bigr|<\infty.
$$
This constant is independent of $z$ and $\nu$. For
$z\in\mathbb C$ and $\nu\in\operatorname{supp}\rho$, we have
\begin{align*}
 |z-\overline{b(\nu)}|&\le |z|+R/2
\end{align*}
and  \begin{equation*}
\left|e^{b(\nu)(z-\overline{b(\nu)})/q}\right|
 =\exp\left(
   \frac{\operatorname{Re}(b(\nu)z)-|b(\nu)|^2}{q}
   \right)
 \le e^{R|z|/(2q)}.
 \end{equation*}
Then the finite sum  in \eqref{E24} satisfies
\begin{align*}
 &\qquad\left|\sum_{\ell=0}^{m+1}\binom{m+1}{\ell}
 \bigl(z-\overline{b(\nu)}\bigr)^{m+1-\ell}
 (\psi^\#)^{(\ell)}\bigl(b(\nu)/q\bigr)\right|\\
 &\quad\le M_m\sum_{\ell=0}^{m+1}\binom{m+1}{\ell}
          (|z|+R/2)^{m+1-\ell}
 =M_m(1+|z|+R/2)^{m+1}.
\end{align*}
Using $1+|z|+R/2\le(1+R/2)(1+|z|)$ in \eqref{E24}, we obtain
\begin{equation}\label{E25}
\begin{aligned}
 |k_{f,m}(z)|
 &\le\frac{M_m e^{R|z|/(2q)}}{q^{m+2}(2\pi)^2}
       (1+|z|+R/2)^{m+1}
       \int_{\mathbb R^2}|\rho(\nu)|\,\dd\nu\\
 &\le C_m(1+|z|)^{m+1}e^{R|z|/(2q)}
 \quad(z\in\mathbb C),
\end{aligned}
\end{equation}
where
$$
 C_m:=\frac{M_m(1+R/2)^{m+1}}{q^{m+2}(2\pi)^2}
       \int_{\mathbb R^2}|\rho(\nu)|\,\dd\nu>0.
$$
The constant $C_m$ depends only on $m$, $c$, $R$ and $\rho$,
and is independent of $z$.

For each fixed $\nu$, the integrand in \eqref{E24} is entire
in $z$. On every disk $|z|\le A$, with $A>0$, the same estimates
bound its absolute value, including the prefactor, by
$$
 \frac{M_m(1+A+R/2)^{m+1}e^{RA/(2q)}}{q^{m+2}(2\pi)^2}
 |\rho(\nu)|\in L^1(\mathbb R^2,\dd\nu).
$$
This local uniform domination shows that $k_{f,m}$ is entire.
Finally, \eqref{E25} and
$R|z|/q\le |z|^2/2+R^2/(2q^2)$ give
$$
\begin{aligned}
 \int_{\mathbb C}|k_{f,m}(z)|^2\,\dd\mu(z)
 &\le\frac{C_m^2}{\pi}\int_{\mathbb C}
 (1+|z|)^{2m+2}e^{-|z|^2+R|z|/q}\,\dd A(z)\\
 &\le\frac{C_m^2e^{R^2/(2q^2)}}{\pi}
 \int_{\mathbb C}(1+|z|)^{2m+2}e^{-|z|^2/2}\,\dd A(z)
 <\infty.
\end{aligned}
$$
Thus, $k_{f,m}\in F^2(\mathbb C)$.
For the orthonormal basis $(e_j)_{j\ge0}$ used in
\eqref{eq:moment-domain}, fix $j\in\mathbb N_0$.
The growth bound \eqref{E1} gives
$fw^m\overline{e}_j\in L^1(d\mu)$ and justifies $j$
differentiations in the parameter $z$ under the defining
integral in $w$ in \eqref{E3}, with $m$ fixed.
Evaluating the resulting identity at $z=0$ gives
$$
 \int_{\mathbb C}f(w)w^m\overline{e_j(w)}\,\dd\mu(w)
 =\frac{k_{f,m}^{(j)}(0)}{\sqrt{j!}}.
$$
These are the coefficients of $k_{f,m}$ in that orthonormal
basis, so they are square summable. Then the domain conditions
\eqref{E9} and the expansion \eqref{eq:moment-domain} give
$w^m\in\mathcal D(\widetilde T_f)$ and
$\widetilde T_f(w^m)=k_{f,m}$. By linearity,
$\mathcal P\subset\mathcal D(\widetilde T_f)$.

Equation \eqref{eq:moment-right-polynomial} gives
$\widetilde T_g1=\gamma^{-1}\psi\ne0$, since $\psi$ is nonzero
by Lemma \ref{L1}. To prove that $\widetilde T_f$ is nonzero,
define $h_0:\mathbb C\to\mathbb C$ by
$$
 h_0(z)=
 \begin{cases}
  \psi(z)/z,&z\ne0,\\
  \psi'(0),&z=0
 \end{cases}
$$
which is entire because $\psi(0)=0$ by Lemma \ref{L1}.
Continuity bounds $h_0$ on the unit disk, and
\eqref{eq:moment-psi-bound} gives
$|h_0(z)|\le C e^{\sigma|z|^2}$ for $|z|\ge1$.
Since $\sigma<1/2$, these bounds imply $h_0\in F^2(\mathbb C)$.
Using \eqref{eq:moment-periodic-weight} and \eqref{eq:moment-symbols}, we obtain
$$
 f(w)h_0(w)e^{-|w|^2}=v(w)r(w).
$$
The identity holds also at $w=0$. Clearly,
$vr\in\mathcal S(\mathbb C)$ by Proposition \ref{L2}. Then \eqref{E19}
shows that $vr$ remains integrable after multiplication by
any polynomial in $w$ and $\overline{w}$.
For each fixed $j\in\mathbb N_0$, apply
\eqref{eq:moment-identities} to the integral
$\int_{\mathbb C}f(w)h_0(w)\overline{e_j(w)}\,\dd\mu(w)$
in $w$. This gives normalized coefficients $v_0/\pi$ for
$j=0$ and zero for $j\ge1$.
The average $v_0$ defined in \eqref{E11} is positive by
Proposition \ref{L2}, so the conditions in \eqref{E9} and the
expansion \eqref{eq:moment-domain} give
\begin{equation}\label{eq:moment-nonzero}
 h_0\in\mathcal D(\widetilde T_f),\quad
 \widetilde T_f h_0=\frac{v_0}{\pi}\ne0.
\end{equation}
Both operators are therefore nonzero in $F^2(\mathbb C)$. The proof is complete.
\end{proof}

Using \eqref{eq:moment-right-polynomial}, we can verify the
zero product by applying the integral identities for $wvr$
in \eqref{E14}.

\begin{proposition}\label{P2}
For the symbols in \eqref{eq:moment-symbols}, the product
$\widetilde T_f\widetilde T_g$ on $F^2(\mathbb C)$ is defined
and zero on $\mathcal P$.
\end{proposition}

\begin{proof}
For $p\in\mathcal P$, equations
\eqref{eq:moment-periodic-weight},
\eqref{eq:moment-symbols}, and
\eqref{eq:moment-right-polynomial} give
$$
 f(w)(\widetilde T_g p)(w)e^{-|w|^2}
 =\gamma^{-1}wv(w)r(w)p(w/\gamma).
$$
Moreover, Proposition \ref{L2} gives $vr\in\mathcal S(\mathbb C)$, so
\eqref{E19} shows that the right side multiplied by any
$\overline{w}^j$ is integrable. Expanding $p(w/\gamma)$ into monomials
and applying \eqref{E14} shows that each of
these integrals is zero.
By \eqref{E9} and \eqref{eq:moment-domain},
$\widetilde T_g p\in\mathcal D(\widetilde T_f)$ and
$\widetilde T_f\widetilde T_g p=0$ in $F^2(\mathbb C)$, proving
\eqref{eq:moment-product}, which completes the proof.
\end{proof}

The preceding propositions complete the operator part of the
construction, so we can combine them with the symbol properties.

\begin{proof}[Proof of Theorem \ref{thm:moment-example}]
For the symbols in \eqref{eq:moment-symbols}, Lemma \ref{L3}
gives smoothness, growth and nonradiality, including the
boundedness of $g$ when $c\ge3/4$. Lemma \ref{P1} proves
that both operators are nonzero in $F^2(\mathbb C)$ and that
their domains contain $\mathcal P$. Proposition \ref{P2} gives
the zero product in \eqref{eq:moment-product}. So, we finish the proof.
\end{proof}



\section{Concluding remarks}

The counterexamples above show that the zero product implication
fails for general symbols. We now ask whether it holds for
pluriharmonic symbols whose entire parts are of exponential type.
Following \cite{Bauer2015}, we consider
pluriharmonic symbols of the form
$$
 f=f_1+\overline{f_2},\quad g=g_1+\overline{g_2},
 \quad f_1,f_2,g_1,g_2\in\Hol(\Cn),
$$
and require all four entire parts to be of exponential type.
This means that there are constants $C,a>0$ such that
$$
 |f_1(z)|+|f_2(z)|+|g_1(z)|+|g_2(z)|
 \le Ce^{a|z|}\quad(z\in\Cn).
$$

The Toeplitz operators with these symbols are densely
defined on the Fock sapce, see \cite{Bauer2015}. We consider the following question on the Fock space.
\begin{question}\label{q:pluriharmonic-zero-product}
Let $n\ge1$, and let $f,g$ have the pluriharmonic
representations and growth bound above. If $T_fT_g=0$
on $F^2(\Cn)$, is it true that $f=0$ or $g=0$ on $\Cn$?
\end{question}

For $n=1$, Bauer, Choe and Koo proved
\cite[Theorem 4.2]{Bauer2015} that
$$
 T_fT_g=T_gT_f=0\quad\Longrightarrow\quad f=0\ \hbox{or}\ g=0
$$
for symbols with these growth conditions. Immediately before
that theorem, they ask whether the hypothesis can be reduced
to $T_fT_g=0$. Question \ref{q:pluriharmonic-zero-product}
includes their question in dimension one and asks for the
same conclusion in every complex dimension.


\begin{thebibliography}{99}
\bibitem{AC2001}
P. Ahern and \v{Z}. \v{C}u\v{c}kovi\'c,
\emph{A theorem of Brown--Halmos type for Bergman space Toeplitz
operators},
J. Funct. Anal. \textbf{187} (2001), 1, 200--210.
\href{https://doi.org/10.1006/jfan.2001.3811}{doi 10.1006/jfan.2001.3811}.

\bibitem{AC2004}
P. Ahern and \v{Z}. \v{C}u\v{c}kovi\'c,
\emph{Some examples related to the Brown--Halmos theorem for the
Bergman space},
Acta Sci. Math. (Szeged) \textbf{70} (2004), 1--2, 373--378.

\bibitem{AV}
A. Aleman and D. Vukoti\'c,
\emph{Zero products of Toeplitz operators},
Duke Math. J. \textbf{148} (2009), 3, 373--403.
\href{https://doi.org/10.1215/00127094-2009-029}{doi 10.1215/00127094-2009-029}.

\bibitem{Bauer2015}
W. Bauer, B. R. Choe and H. Koo,
\emph{Commuting Toeplitz operators with pluriharmonic symbols
on the Fock space},
J. Funct. Anal. \textbf{268} (2015), 10, 3017--3060.
\href{https://doi.org/10.1016/j.jfa.2015.03.003}{doi 10.1016/j.jfa.2015.03.003}.

\bibitem{BL}
W. Bauer and T. Le,
\emph{Algebraic properties and the finite rank problem for Toeplitz
operators on the Segal--Bargmann space},
J. Funct. Anal. \textbf{261} (2011), 9, 2617--2640.
\href{https://doi.org/10.1016/j.jfa.2011.07.006}{doi 10.1016/j.jfa.2011.07.006}.

\bibitem{BH}
A. Brown and P. R. Halmos,
\emph{Algebraic properties of Toeplitz operators},
J. Reine Angew. Math. \textbf{213} (1964), 89--102.
\href{https://doi.org/10.1515/crll.1964.213.89}{doi 10.1515/crll.1964.213.89}.

\bibitem{Grafakos2014}
L. Grafakos,
\emph{Classical Fourier Analysis},
Third Edition, Graduate Texts in Mathematics, vol. 249,
Springer, New York, 2014.
\href{https://doi.org/10.1007/978-1-4939-1194-3}{doi 10.1007/978-1-4939-1194-3}.

\bibitem{Igusa2000}
J.i. Igusa,
\textit{An Introduction to the Theory of Local Zeta Functions},
AMS/IP Studies in Advanced Mathematics, vol. 14,
American Mathematical Society, Providence, RI;
International Press, Cambridge, MA, 2000.

\bibitem{JohnParthasarathy2021}
T. C. John and K. R. Parthasarathy,
\emph{A common parametrization for finite mode Gaussian states,
their symmetries, and associated contractions with some applications},
J. Math. Phys. \textbf{62} (2021), 022102.

\bibitem{Knapp2016}
A. W. Knapp,
\emph{Basic Real Analysis},
Digital Second Edition,
Published by the Author, East Setauket, NY, 2016.

\bibitem{Le2009}
T. Le,
\emph{Finite-rank products of Toeplitz operators in several
complex variables},
Integral Equations Operator Theory \textbf{63} (2009), 4, 547--555.
\href{https://doi.org/10.1007/s00020-009-1661-6}{doi 10.1007/s00020-009-1661-6}.

\bibitem{Le2010}
T. Le,
\emph{A refined Luecking's theorem and finite-rank products of
Toeplitz operators},
Complex Anal. Oper. Theory \textbf{4} (2010), 2, 391--399.
\href{https://doi.org/10.1007/s11785-009-0008-2}{doi 10.1007/s11785-009-0008-2}.

\bibitem{LeTikaradze2022}
T. Le and A. Tikaradze,
\emph{A generalization of the Brown--Halmos theorems for the unit ball},
Adv. Math. \textbf{404} (2022), 108411.
\href{https://doi.org/10.1016/j.aim.2022.108411}{doi 10.1016/j.aim.2022.108411}.

\bibitem{DLMF}
National Institute of Standards and Technology,
\emph{NIST Digital Library of Mathematical Functions},
Sections 1.9 and 20.2.

\bibitem{DasNarayanan2025}
S. Das and E. K. Narayanan,
\emph{Zero products of Toeplitz operators on the Hardy and
Bergman spaces over an annulus},
New York J. Math. \textbf{31} (2025), 984--1001.

\end{thebibliography}
\end{document}